\documentclass[12pt, reqno]{amsart}
\usepackage[english]{babel}
\usepackage[utf8]{inputenc}
\usepackage[T1]{fontenc}
\usepackage{amsmath}
\usepackage{amsthm}
\usepackage{amssymb}
\usepackage{graphics}
\usepackage{xcolor}
\usepackage{enumitem}
\usepackage[pagebackref, colorlinks = true, linkcolor = teal, urlcolor  = teal, citecolor = violet]{hyperref}
\usepackage[margin=1in]{geometry}
\usepackage{cleveref}
\crefname{equation}{Eq.}{Eqs.}
\usepackage{bbm}
\usepackage{pifont}
\usepackage{tikz}
\usepackage{graphicx}
\usepackage{bbold}
\usepackage{hyperref}

\usepackage[normalem]{ulem}

\newtheorem{theorem}{Theorem}[section]
\newtheorem{definition}[theorem]{Definition}
\newtheorem{proposition}[theorem]{Proposition}

\newtheorem{lemma}[theorem]{Lemma}
\newtheorem{remark}[theorem]{Remark}

\newcommand{\id}{{\rm Id}}

\newcommand{\one}{{\mathbf 1}}
\newcommand{\tr}{{\operatorname{tr}}}

\renewcommand{\d}{{\rm d}}

\newcommand{\pp}{{\mathbb P}}
\newcommand{\ee}{{\mathbb E}}
\newcommand{\ppch}{\pp^{\mathrm{ch}}}

\newcommand{\rr}{{\mathbb R}}
\newcommand{\nn}{{\mathbb N}}
\newcommand{\cc}{{\mathbb C}}

\renewcommand{\P}{{\mathrm P}}

\newcommand{\mat}{\mathrm{M}_k(\cc)}

\newcommand{\supp}{\operatorname{supp}}

\newcommand{\as}{\operatorname{-}\mathrm{a.s.}}

\newcommand\pt[1]{^{\otimes #1}}

\newcommand{\sta}{{\P(\cc^k)}}

\newcommand{\staalg}{\mathcal{B}}
\newcommand{\out}{\Omega}
\newcommand{\outalg}{\mathcal{O}}
\newcommand{\joint}{\sta\times \out}
\newcommand{\jointalg}{\mathcal{J}}

\renewcommand{\Re}{\operatorname{Re}}

\makeatletter
\newcommand{\specialcell}[1]{\ifmeasuring@#1\else\omit$\displaystyle#1$\ignorespaces\fi}
\makeatother

\begin{document}

\title{Quantum Trajectories. Spectral Gap, Quasi-compactness $\&$ Limit Theorems.}
\author{Tristan Benoist}
\email{tristan.benoist@math.univ-toulouse.fr}
\address{Institut de Mathématiques, UMR5219, Université de Toulouse, CNRS, UPS, F-31062 Toulouse Cedex 9, France}
 
\author{Arnaud Hautecœur}
\email{arnaud.hautecoeur@math.univ-toulouse.fr}
\address{Institut de Mathématiques, UMR5219, Université de Toulouse, CNRS, UPS, F-31062 Toulouse Cedex 9, France}

\author{Clément Pellegrini}
\email{clement.pellegrini@math.univ-toulouse.fr}
\address{Institut de Mathématiques, UMR5219, Université de Toulouse, CNRS, UPS, F-31062 Toulouse Cedex 9, France}
\maketitle

\begin{abstract}
   Quantum trajectories are Markov processes modeling the evolution of a quantum system subjected to repeated independent measurements. Inspired by the theory of random products of matrices, it has been shown that these Markov processes admit a unique invariant measure under a purification and an irreducibility assumptions. This paper is devoted to the spectral study of the underlying Markov operator. Using Quasi-compactness, it is shown that this operator admits a spectral gap and the peripheral spectrum is described in a precise manner. Next two perturbations of this operator are studied. This allows to derive limit theorems (Central Limit Theorem, Berry-Esseen bounds and Large Deviation Principle) for the empirical mean of functions of the Markov chain as well as the Lyapounov exponent of the underlying random dynamical system. 
\end{abstract}

\tableofcontents
\newpage

\section{Introduction}

Quantum trajectories are Markov processes describing the evolution of quantum systems undergoing indirect measurements. These physical models have been instrumental in the recent advances in quantum optics. From a physical point of view, their study is motivated by both practical applications and fundamental inquiries, as they delve into the central aspects of measurement and collapse phenomena in quantum mechanics. Among the most remarkable groundbreaking experiments, one can mention the one of the Serge Haroche's group (see for example \cite{guerlin_progressive_2007,haroche2006exploring}) or David Wineland's result \cite{https://doi.org/10.1002/anie.201303404,RevModPhys.75.281} which have been recognized by a Nobel prize in physics. The physics literature is flourishing and we refer to \cite{carmichael,haroche2006exploring,wisemanmilburn,breuer2002theory} for some introduction to the subject. Drawing an analogy with classical physics models like the Langevin equation that portray generic systems perturbed by some environment, quantum trajectories provide insight into the evolution of a quantum system influenced by both environment and measurement backaction. Their long time behaviour linked with their ergodic properties and stationary regimes is of central interest. 

From a mathematical point of view, a quantum trajectory is a Markov chain $(\hat{x}_n)$ on the projective space $\P(\mathbb{C}^k)$, defined by 
\[
\hat x_{n+1}=V_n\cdot \hat x_{n}=V_{n}\ldots V_{1} \cdot \hat x_0,
\]
where $V_n$ is a $\textrm M_k(\mathbb C)$-valued random variable with law $ ||v x_n||^2 \d \mu(v)$ and for $A\in \mathrm M_k(\cc)$ and $x\in \cc^k$ such that $Ax\neq 0$, $A\cdot\hat x$ is the element of $\sta$ representing $Ax$. The measure $\mu$ is a measure on $\mathrm M_k(\mathbb C)$ that satisfies
\begin{equation}
\int_{\mathrm M_k(\mathbb C)} v^* v \,\mathrm{d} \mu(v) = \id_{\cc^k}. 
\end{equation}
The associated transition kernel is given by
\begin{equation} \label{eq_deftranskernel1}
\Pi f(\hat x)=\int_{\mathrm M_k(\cc)} f(v\cdot \hat x)  \|v x\|^2 \d \mu(v),
\end{equation}
where $f$ is any bounded and measurable function. The study of these processes is particularly rich in intriguing mathematical problems. Notably, the transitions of the associated Markov chains are singular, and they do not exhibit generic $\varphi$-irreducibility \cite[Proposition~8.1]{BFP23}. Consequently, standard Markov technologies, as outlined in \cite{MT}, can not be applied directly. Another perspective is to view them as so-called place-dependent iterated function systems (IFS); however, they diverge from the standard criteria of contractivity for IFS (refer to \cite{barnsley_invariant_1988} for conventional IFS results). 

A particular interesting approach is portrayed by the analogy with the theory of random products of matrices. This theory was pioneered by Kesten and Furstenberg in \cite{furstenberg1960products}. The theoretical framework considers sequences of independent and identically distributed (i.i.d.) matrices $(A_i)$, with a primary emphasis on the top Lyapunov exponent, as discussed in \cite{furstenberg1960products,le2006theoremes,aoun2021law}: there exists $\gamma\in [-\infty,\infty[$ such that

\begin{equation}\label{eq:lyapclass}
\gamma=\lim_{n}\frac{1}{n}\log\Vert A_n \ldots A_1\Vert, \quad\mbox{almost surely}.
\end{equation} 
This result notably constitutes a generalization of the law of large numbers within a non-commutative context. In order to obtain results in this field, one is particularly interested in the study of the Markov chain $(\hat x_n)$, defined by $\hat x_{n+1} = A_{n+1}\cdot \hat x_n$. This is the essence of the analogy with quantum trajectories. Comprehensive insights into this field are available in reference books such as \cite{BL85} and \cite{Quint}. Numerous contributions have addressed limit theorems pertaining to the random products of matrices, with the law of large numbers intricately connected to \eqref{eq:lyapclass}. Subsequent refinements, including the central limit theorem, have been addressed in \cite{hennion1997limit, BL85, xiao2021limit, benoist2016central}. For analyses involving large deviations, results are presented in \cite{buraczewski2016precise, sert2019large, xiao2022edgeworth,gol1989lyapunov,10.1007/BFb0076833}.

A particularly robust methodology for deriving limit theorems involves the examination of the invariant measures and spectral properties of associated Markov chains. Beyond the simple analogy of the random products of matrices and quantum trajectories, the strategies and tools developed in this field are particularly pertinent to the study of the long-term behavior of quantum trajectories. Nevertheless, while quantum trajectories share a strong connection with the theory of products of random matrices, there are crucial differences that render their study particularly intriguing and involved. The main difference is that the random matrices $(V_{n})_n$ are non-i.i.d. Their law is defined by $||v x_n||^2\d\mu(v)$. Thanks to this alternative law that depends on the initial state $\hat x_0$, in contrast to the conventional theory of random products of matrices where matrices are necessarily invertible, this requirement is not mandatory in the context of quantum trajectories. This fundamental distinction render the theory of products of i.i.d. matrices not sufficient in our context but leads to results requiring weaker assumptions on $\mu$. 

 In \cite{BFPP17}, given a purification assumption (equivalent to the typical contractivity assumption in the theory of products of random matrices) and an irreducibility assumption (weaker than the commonly required strong irreducibility assumption), it has been demonstrated that the Markov operator defined in Equation~\eqref{eq_deftranskernel1} possesses a unique probability invariant measure $\nu_{\text{inv}}$ (\emph{i.e.} satisfying $\nu_{\text{inv}}\Pi = \nu_{\text{inv}}$). Subsequently, by introducing a statistical estimator of $(\hat x_n)$ based on maximum likelihood, exponential convergence toward the invariant measure, measured in Wasserstein distance, has been established. 
 
 The next natural step is to study the spectral properties of the Markov transition operator $\Pi$. This is one of the main aim of this article. Of particular significance in this context is the quasi-compactness of $\Pi$. We refer to \cite{HeHe01} and references therein for a comprehensive presentation of quasi-compact operators, we refer to \cite{Gui, 10.1214/13-AIHP566} for some applications and inspiration in the field of random products of matrices. Here, we prove that the operator $\Pi$ is quasi-compact on Banach spaces of H\"older continuous functions. Within this approach, we make explicit the peripheral spectrum of the operator $\Pi$ and we provide the corresponding eigenfunctions. Next, we study two different types of perturbations of $\Pi$. These perturbations are tiltings as used in some proofs of large deviation principle and limit theorems for Markov chains. We similarly obtain limit theorems for the empirical mean and the top Lyapounov exponent of the quantum trajectories. Such approach follows the usual one for random products of matrices \cite{BL85,HeHe01} but the specificity of our model imposes to develop new estimations and new arguments. In particular, the fact that the matrices $(V_n)$ are not invertible prevents exploitation of the known results from random products of matrices. Furthermore, since strong irreducibility is not assumed, several results on the support of invariant measures cannot be used and several key steps have to be completely re-investigated. This is particularly evident for the tilting analysis.
 
 The tilting method allows us to derive central limit theorems and large deviation principles. These limit theorems complete the ones obtained in \cite{BFP23}. The approach in \cite{BFP23} exploited the geometric convergence towards the invariant measure to construct solutions to Poisson equations and then martingale limit theorems. Here, all the limit theorems are consequences of the spectral gap property and the operator perturbation theory. Unlike \cite{BFP23}, we derive a large deviation principle as a restricted variant of G\"artner-Ellis Theorem (only moderate deviation principle were presented in \cite{BFP23}). Here, as in \cite{BFP23}, we derive central limit theorems for the empirical mean but we extend it also to the Lyapounov exponent. Furthermore, we strengthen them by deriving Berry-Esseen type results which are out of the scope of \cite{BFP23}.
\bigskip

The paper is structured as follows. In Section~\ref{sec:def}, we properly define quantum trajectories, introducing the Markov operator $\Pi$ and its tiltings. Section~\ref{sec:results} presents the main results of the paper, focusing on the spectral properties of the operator $\Pi$ and the analyticity of its perturbations. Additionally, we detail various limit theorems derived from the perturbation results. Section~\ref{sec:proof QCompact} is dedicated to proving the quasi-compactness of the operator $\Pi$ and making precise the whole peripheral spectrum. Following that, Section~\ref{sec:proofso} provides the proof the perturbations are analytic. In Section~\ref{sec:limitprooof}, we finally present the proofs of the limit theorems, specifically central limit theorems, speed of convergence (Berry-Esseen bounds) and restricted large deviation principles. Several technical tools are deferred to the appendices.

\section{Definitions}\label{sec:def}
This section introduces our main objects. We mostly use the definitions and notations of \cite{BFPP17}. We use the convention, $\nn=\{1,2,\dotsc\}$.

For $k\in \nn$ fixed, we consider the complex vector space $\mathbb{C}^k$ as a Hilbert space with the canonical inner product and denote by $\|.\|$ the $2$-norm on $\cc^k$. Our main process takes value in the  projective space $\P(\mathbb{C}^k)$. We equip it with its Borel $\sigma$-algebra $\staalg$.
For a nonzero vector $x\in\mathbb C^k$, we denote $\hat x$ the corresponding equivalence class of $x$ in $\P(\mathbb C^k)$. Accordingly, given $\hat x$, we denote by $x$ one of its representative of unit norm. We equip $\sta$ with the metric
$$d(\hat x,\hat y)=\sqrt{1-|\langle x,y\rangle|^2}.$$

For $0<\alpha\leq 1$, $C^{\alpha}(\sta,\cc)$ is the space of $\alpha$-Hölder continuous functions of $\sta$ valued in $\mathbb C$ (we shall also need sometimes to consider $C^{\alpha}(\sta,\mathbb R)$). We consider it as a Banach space with norm 
$$\|f\|_\alpha=\|f\|_\infty+m_\alpha(f)$$
for any function $f\in C^\alpha(\sta,\cc)$ where $\|.\|_\infty$ is the sup norm and $m_\alpha(f)$ is $f$ Hölder coefficient. The spectral radius of a bounded endomorphism $L$ of $C^\alpha(\sta,\cc)$ is denoted $\rho_\alpha(L)$.
\bigskip

We study a Markov chain on $\sta$ through the spectral analysis of its transition kernel $\Pi$, and perturbations of it, as linear maps on $C^\alpha(\sta,\cc)$. For a matrix  $v\in \mathrm{M}_k(\mathbb C)$ we denote $v\cdot \hat{x}$ the element of the projective space represented by $v\,x$ whenever $v\,x\neq 0$. We equip $\mathrm{M}_k(\mathbb C)$ with its Borel $\sigma$-algebra and let $\mu$ be a measure on $\mathrm{M}_k(\mathbb C)$ with a finite second moment, $\int_{\mathrm{M}_k(\mathbb C)} \|v\|^2\,\d\mu(v)<\infty$, that satisfies the stochasticity condition
\begin{equation}
\label{eq:stochastic family}
\int_{\mathrm{M}_k(\mathbb C)} v^* v \,\mathrm{d} \mu(v) = \id_{\cc^k},
\end{equation}
as already presented in Introduction. We define the Markov chain $(\hat{x}_n)$  on $\P(\mathbb{C}^k)$ as follows :
\[
\hat x_{n+1}=V_n\cdot \hat x_{n},
\]
where $V_n$ is an $\mathrm{M}_k(\mathbb C)$-valued random variable with law knowing $\hat{x}_n$, $ \|v x_n\|^2 \mathrm{d} \mu(v)$. Condition \eqref{eq:stochastic family} ensures this is a probability measure for any $\hat x_n\in \sta$. More precisely, such a Markov chain is associated with the transition kernel given by
\begin{equation} \label{eq_deftranskernel}
\Pi f(\hat x)=\ee(f(\hat x_1)|\hat x_0=\hat x)=\int_{\mathrm{M}_k(\cc)} f(v\cdot \hat x)  \|v x\|^2 \mathrm{d} \mu(v),
\end{equation}
{
for any continuous function $f:\sta\to\cc$ and $\hat x\in \P(\mathbb C^k)$. Since $f$ is continuous on $\sta$, whenever $\|v x\|^2=0$, we put $f(v\cdot \hat x)  \|v x\|^2=0$. With this convention, the function $\hat x\mapsto f(v\cdot \hat x)  \|v x\|^2 $ remains continuous.}

Now, since$$ \int_{\mathrm{M}_k(\cc)}\|vx\|^2\d\mu(v)=1,$$
the operator $\Pi$ is then well-defined as a linear map from bounded measurable functions to bounded measurable functions. 

For any probability measure $\nu$, we denote $\nu\Pi$ the probability measure defined by 
$$\nu\Pi(f)=\mathbb E_\nu[f(\hat x_1)]=\int_{\P(\cc^k)}\Pi f(\hat x)\d \nu(\hat x),$$
for any continuous function $f$. A measure $\nu$ is called invariant if $\nu \Pi = \nu$.

We shall also need to consider another Markov chain $(V_n,\hat x_n)_n$ where we keep the memory of the matrix $V_n$ inducing the transition $\hat x_{n-1}$ to $\hat x_n$. To this end we follow the construction of \cite{BFPP17}. 
We consider the space of infinite sequences $\out:=\mathrm{M}_k(\cc)^{\mathbb N}$, write $\omega = (v_1,v_2, \dots)$ for any such infinite sequence, and denote by $\pi_n$ the canonical projection on the  first $n$ components, $\pi_n(\omega)=(v_1,\ldots,v_n)$. Let $\mathcal M$ be the Borel $\sigma$-algebra on $\mathrm{M}_k(\cc)$. For $n\in\nn$, let $\outalg_n$ be the  $\sigma$-algebra on $\Omega$ generated by the $n$-cylinder sets, i.e.\ $\outalg_n = \pi_n^{-1}(\mathcal M^{\otimes n})$. We equip the space $\Omega$ with the smallest $\sigma$-algebra $\outalg$ containing $\outalg_n$ for all $n\in \nn${, \emph{i.e.} with its product $\sigma$-algebra}. We let $\mathcal B$ be the Borel $\sigma$-algebra on $\P(\mathbb C^k)$, and denote
\[\jointalg_n=\mathcal B\otimes \outalg_n,\qquad \jointalg=\mathcal B\otimes \outalg.\]
This makes $\big(\P(\mathbb C^k)\times \Omega,\jointalg\big)$ a measurable space. With a small abuse of notation, we denote the sub-$\sigma$-algebra $\{\emptyset,\P(\mathbb C^k)\}\times \outalg$ by $\outalg$, and equivalently identify any $\outalg$-measurable function $f$ with the $\jointalg$-measurable function $f$ satisfying $f(\hat{x},\omega) = f(\omega)$.

For $i\in\mathbb N$, we consider the random variables $V_i : \Omega \to  \mathrm{M}_k(\mathbb C)$,
\begin{equation}
	V_i(\omega) = v_i \quad \mbox{for} \quad \omega=(v_1,v_2,\ldots), \label{eq:W}
\end{equation}
and we introduce $\mathcal O_n$-mesurable random variables $(W_n)$ defined for all $n\in\mathbb N_0$ as $W_0=\id$,
$$W_n=V_{n}\ldots V_{1}.$$ 

With a small abuse of notation, we identify cylinder sets and their bases, and extend this identification to several associated  objects. In particular, we identify $O_n\in \mathcal M^{\otimes n}$ with $\pi_n^{-1}(O_n)$, a function $f$ on $\mathcal M^{\otimes n}$  with $f \circ \pi_n$ and a measure $\mu^{\otimes n}$ with the measure $\mu^{\otimes n} \circ \pi_n$.
Since $\mu$ is not necessarily finite, we can not extend $(\mu^{\otimes n})$ into a measure on $\Omega$.

Let $\nu$ be a probability measure on $(\sta,\mathcal B)$. We extend it to a probability measure $\mathbb{P}_\nu$ on $(\sta\times\out,\jointalg)$ by letting, for any $S\in \mathcal B$ and any cylinder set $O_n \in \outalg_n$,
\begin{equation} \label{eq_defPnu}
\mathbb{P}_\nu(S \times O_n):=\int_{S\times O_n}  \|W_n(\omega)x\|^2 \d\nu(\hat x) \d\mu\pt n(\omega).
\end{equation}
From relation \eqref{eq:stochastic family}, it is easy to check that the expression \eqref{eq_defPnu} defines a consistent family of probability measures and, by Kolmogorov's extension Theorem, this defines a unique probability measure $\pp_\nu$ on $\joint$. In addition, the restriction of $\mathbb{P}_\nu$ to $\mathcal B\otimes \{\emptyset,\Omega\}$ is by construction $\nu$.
\bigskip

We can consider the random process $(V_n, \hat x_n)$ whose transition is given by
$$\mathbb P[V_{n+1}\in A,\hat x_{n+1}\in B\vert \hat x_n=x]=\int_A\mathbb 1_B(v\cdot \hat x)\Vert vx\Vert^2\d\mu(v),$$
for all Borel sets $A$ of $\mathrm M_k(\mathbb C)$ and all Borel sets $B$ of $\sta$. The process $(\hat x_n)$ on $(\Omega\times \P(\mathbb C^k),\jointalg, \mathbb{P}_\nu)$ has the same distribution as the Markov chain defined by $\Pi$ and initial probability measure $\nu$. The process $(V_n,\hat x_n)_n$ will be usefull in Section~\ref{sec:results W_n} and more precisely in the proofs of the results of the Lyapounov exponent. The process $(W_n)$ will also be used in Section~\ref{sec:results W_n}.

Note that the couple $(V_n,\hat x_n)$ is the original Markov process $(\hat x_n)$ with the memory of which matrix $V_n$ induced the transition from $\hat x_{n-1}$ to $\hat x_{n}$. It is defined only for $n\geq 1$. Note that on its own the sequence $(V_n)$ is not a Markov chain. Actually it is a hidden Markov chain with underlying Markov chain $(\hat x_n)$. This fact differs crucially from the usual theory of random products of matrices where the choice of matrices is i.i.d, thus Markov.


\bigskip
We are now in the position to present the two classes of tiltings of $\Pi$ so as to derive some limit theorems for the Markov chain $(\hat x_n)$. First, for any $h\in C^\alpha(\sta,\mathbb R)$ and $z\in \cc$, let $\Pi_z$ be the linear map on $C^\alpha(\sta,\cc)$ defined for all $\hat x\in \sta$
$$\Pi_zf(\hat x)=\int e^{zh(v\cdot\hat x)}f(v\cdot\hat x)\|vx\|^2\d\mu(v)=\Pi(e^{zh}f)(\hat x).$$
Remark that $\Pi_0=\Pi$. This first tilting (we shall use also the name perturbation from time to time) is related  to the Markov chain $(\hat x_n)$.

Second, assuming there exist $\tau\in (0,2)$ and $\delta>0$ such that $\int \|v\|^{2+s}\d\mu(v)<\infty$ for any $s\in (-\tau,\delta)$, for any $z$ in the strip $\{w\in \cc: \operatorname{Re}(w)\in (-\tau,\delta)\}$, let $\Gamma_z$ be the linear map on $C^\alpha(\sta,\cc)$ defined by
$$\Gamma_zf(\hat x)=\int f(v\cdot \hat x)e^{z\log \|vx\|}\|vx\|^2\d\mu(v)$$
for all $\hat x\in \sta$.
Again, $\Gamma_0=\Pi$. This tilting is related to the Markov chain $(V_n,\hat x_n).$

The tiltings $\Pi_z$ and $\Gamma_z$ are particularly useful to study limit theorems. Proofs of limit theorems using that method require some smoothness with respect to $z$ in a neighborhood of $0$. This is what we will prove on top of the spectral gap for $\Pi$.

\section{Results}\label{sec:results}
\subsection{Assumptions}
We prove our results under the same two assumptions as in \cite{BFPP17,BFP23}. Their motivation and meaning are discussed in both these references. Especially, in \cite{BFPP17} a comparison with usual assumptions for i.i.d. products of matrices is provided. One notable difference with the i.i.d. situation is that we do not assume $v$ is invertible $\mu$-almost everywhere. 

The first one is a purification hypothesis. It means that there is no subspaces of $\cc^k$ of dimension greater or equal to $2$ on which product of matrices from the support of $\mu$ are proportional to unitary matrices. It is akin to the contractivity assumption for i.i.d. products of matrices -- see \cite[Appendix A]{BFPP17}.

\medskip
\noindent{\bf (Pur)} Any orthogonal projector $\pi$ such that $\pi v_1^*\dotsb v_n^*v_n\dotsb v_1\pi\propto \pi$ for $\mu^{\otimes n}$-almost every $(v_1,\dotsc,v_n)$ is of rank one.
\medskip

Here the symbol $\propto$ means "proportional to" with $0$ an allowed proportionality constant. The second assumption concerns the irreducibility of the map $\Phi: X\mapsto \int v^*Xv\d\mu(v)$ -- see \cite{BFPP17,BFP23}. It is weaker than the strong irreducibility assumption used for i.i.d. products of matrices. Note that the map $\Phi$ is called a quantum channel in the quantum mechanics and information communities.

\medskip
\noindent{\bf (Erg)} There exists a unique minimal subspace $E\neq\{0\}$ of $\cc^k$ such that $vE\subset E$ for $\mu$-almost every $v$.
\medskip

All our results are formulated assuming {\bf (Pur)} and {\bf (Erg)} hold. We recall them when useful but do not repeat them systematically.

\subsection{Spectral gap}
{ As a preliminary, we ensure $\Pi$ is a bounded endomorphism of the Hölder function Banach space.
\begin{proposition}\label{prop:Endo}
For every $0<\alpha\leq 1$, $\Pi$ is a bounded endomorphism of $C^\alpha(\sta,\cc)$.
\end{proposition}
The proof of this proposition is a direct corollary of \Cref{Lemma}. Our first main result is that, actually, $\Pi$ is quasi-compact.} 
In particular it possesses a spectral gap.

\begin{theorem}\label{QCompact}
    Assume {\bf (Pur)} and {\bf (Erg)} hold. Then, for every $0< \alpha \leq 1 $, there exist two subspaces $\mathcal{F}$ (independent of $\alpha$) and $\mathcal{G}_\alpha$ of $C^\alpha(\sta,\cc)$ such that 
    \begin{enumerate}
        \item \label{Compact1} $C^\alpha(\sta,\cc) = \mathcal{F} \oplus \mathcal{G}_\alpha $,
        \item \label{Compact1bis} $\Pi \mathcal F\subset \mathcal F$ and $\Pi \mathcal G_\alpha\subset \mathcal G_\alpha$,
        \item \label{Compact3}$\rho_\alpha(\Pi|_{\mathcal{G}_\alpha}) < \rho_\alpha(\Pi)=1 $,
        \item \label{Compact2}$\mathcal{F}$ is finite dimensional and the spectrum of $\Pi|_{\mathcal{F}}$ is a finite subgroup of $U(1)=\{z\in\cc,\vert z\vert=1\}$:
        $$\sigma(\Pi|_{\mathcal{F}}) =\{ e^{i\frac{l}{p}2\pi} \; | \; l=0,...,p-1 \}.$$
        In addition, all these eigenvalues are simple.
    \end{enumerate}
\end{theorem}
In summary, if {\bf (Pur)} and {\bf (Erg)} hold, $\Pi$ can be written as,
\begin{equation}\label{eq:decompoquasicompact}\Pi=\quad\sum_{l=0}^{p-1} e^{i\frac{l}{p}2\pi}f_l\nu_l \quad +\quad T,
\end{equation}
in the sense that for all bounded and measurable functions $f$, we have
$$\Pi f=\quad\sum_{l=0}^{p-1} e^{i\frac{l}{p}2\pi}f_l\nu_l(f) \quad +\quad T(f).$$
Here, we have that 
$$f_l\in C^\alpha(\sta,\cc),\quad \nu_l\in C^\alpha(\sta,\cc)^*\quad \mbox{and}\quad \nu_l(f_k)=\delta_{l,k}$$
and 
$$T:C^\alpha(\sta,\cc)\to C^{\alpha}(\sta,\cc),\quad \rho_\alpha(T)<1\quad and \quad Tf_l=0,\quad \nu_lT=0$$
for any $l,k\in \{0,\dotsc,p-1\}$.

Note that along the proof of part (4) of Theorem~\ref{QCompact}, we give the explicit form of the functions $f_l,l=0,\ldots,p-1$. They are obtained in Lemma~\ref{lem:charactrization of cF}. This gives a complete characterization of the space $\mathcal F$ as the linear span of those functions. These eigenfunctions are defined in terms of eigenvectors of the linear operator $\Phi$ associated to $\mu$ whose study is postponed to Appendix~\ref{app:cycl}.

 Theorem~\ref{QCompact} is proved in Section~\ref{sec:proof QCompact}. It is based on a generalization of a theorem by Ionescu-Tulcea and Marinescu~\cite{ITM50} that can be found in \cite{HeHe01}. These tools were used to prove quasi-compactness for i.i.d. products of matrices -- see \cite{Gui} and references therein.

\subsection{Tiltings}
The two tiltings we have defined are analytic with respect to $z$.
\begin{theorem}\label{Analytique}
Assume {\bf (Pur)} and {\bf (Erg)} hold. Then, for any $0<\alpha\leq 1$ {and $z \in \cc$, the operator $\Pi_z$ is a bounded endomorphism of $C^\alpha(\sta,\cc)$ and the map} 
$$
\begin{array}{rcl}
    z & \longmapsto & \Pi_z
\end{array}
$$
is analytic on $\cc$.
\end{theorem}
\begin{theorem}\label{AnalyticLog}
    Assume {\bf (Pur)} and {\bf (Erg)} hold. Assume there exist $\tau\in (0,2)$ and $\delta>0$ such that $\int \|v\|^{2-\tau}+\|v\|^{2+\delta}\d\mu(v)<\infty$. Let $\alpha=1-\tau/2$. Then, {for every $z \in \{z\in\mathbb C:\Re(z)\in(-\tau,\delta)\}$, the operator $\Gamma_z$ is a bounded endomorphism of $C^\alpha(\sta,\cc)$ and the map}
    $$\begin{array}{ccccc}
              z & \longmapsto & \Gamma_z
    \end{array}$$
    is analytic on the strip $\{z\in\mathbb C:\Re(z)\in(-\tau,\delta)\}$.
\end{theorem}
{
The proof of Theorem~\ref{Analytique} is provided in Section~\ref{sec:proof Analytique} and the one of Theorem~\ref{AnalyticLog} in Section~\ref{sec:proof AnalyticLog}. 

\begin{remark}Note that the dual space of $C^\alpha(\sta,\cc)$ is non trivial. Thus, proving a weak analyticity to deduce these results does not seem to be a good strategy. In particular results like \cite[Chapter 3]{kato2013perturbation} cannot be straightforwardly applied. We use two different aternative strategies for the two theorems. The proof of Theorem~\ref{Analytique} relies on an explicit norm convergent series expansions in $z$. Concerning \Cref{AnalyticLog}, the proof relies on \cite[Theorem 3.1]{arendt2000vector}. Using this theorem, we bypass the full characterization of the dual space of $C^\alpha(\sta,\cc)$. Note that Theorem \ref{Analytique} follows also from \cite[Theorem 3.1]{arendt2000vector} but the proof by expansion is straightforward. A proof of \Cref{AnalyticLog} by norm convergent series expansion in $z$ is also possible. Alas, it involves tedious computations and estimations due to the presence of the logarithm and the fact the matrices may be non invertible.
\end{remark}}
These strong smoothness properties of both these perturbations imply some limit theorems for the Markov chain $(\hat x_n)$.
\subsection{Limit theorems}
In \cite{BFPP17}, a major result was the uniqueness of the invariant measure that we recall here.

\begin{theorem}
    Assume {\bf (Pur)} and {\bf (Erg)} hold, there exists a unique probability measure $\nu_{inv}$ on $\sta$ such that $\nu_{inv}\Pi=\nu_{inv}$.
\end{theorem}
In \cite{BFP23} different limit theorems following from the results of \cite{BFPP17} were proved. Using Theorems~\ref{Analytique} and~\ref{AnalyticLog}, we can prove different type of limit theorems. The only theorem already proved in \cite{BFP23} is the central limit theorem for the empirical mean of quantum trajectories. Here we complement it with a Berry-Esseen bound quantifying the rate of convergence.

In all these theorems, $\mathcal N(0,\sigma^2)$ is the probability measure of a normal centered random variable of variance $\sigma^2$. If $\sigma^2=0$, $\mathcal N(0,\sigma^2)$ stands for the Dirac measure in $0$. If $X\sim \mathcal N(0,\sigma^2)$, the notation $\mathcal N(0,\sigma^2)(\phi)$ will stand for $\mathcal N(0,\sigma^2)(\phi)=\mathbb E[\phi(X)]$, for any integrable functions $\phi$. We shall also use 
$\mathcal{N}(0,\sigma^2)((-\infty,u])=\mathbb P[X\leq u]$.

\subsubsection{Quantum trajectory empirical mean}
For any function $h:\sta\to\rr$ and $n\in\mathbb N$, let
$$S_n(h)=\sum_{l=1}^n h(\hat x_l).$$
Then, for any $h$ Hölder continuous function, $(S_n(h))$ verifies a central limit theorem with a speed of convergence given by a Berry-Esseen bound.
\begin{theorem}[Central Limit Theorem]\label{thm:clt1}
    Assume {\bf (Pur)} and {\bf (Erg)} hold. Let $h\in C^\alpha(\sta,\mathbb R)$ such that $\ee_{\nu_{inv}}[h]=0$. Then there exists $\sigma^2 \geq 0$ such that:
    \begin{enumerate}
        \item For every $\phi \in C_b(\rr)$ and every probability measure $\nu$ on $\sta$,
        $$ \underset{n \to \infty}{\lim} \; \left | \ee_{\nu} \left [\phi \left ( \frac{S_n(h)}{\sqrt{n}} \right ) \right] - \mathcal{N}(0,\sigma^2) (\phi) \right| = 0,$$
        \item $\sigma^2 = \underset{n}{\lim} \; \frac{1}{n}\ee_{\nu_{inv}}[S_n(h)^2].$
        \item If $\sigma^2 >0$, there exists $C>0$ such that for every probability measure $\nu$ on $\sta$:
        $$ \underset{u \in \rr}{\sup} \; | \pp_{\nu}[S_n \leq u \sigma \sqrt{n}] - \mathcal{N}(0,1)((-\infty,u])| \leq  \frac{C}{\sqrt{n}}.$$
    \end{enumerate}
\end{theorem}
The proof of the theorem is a direct consequence of Theorem~\ref{Analytique} and \cite[Theorems A and B]{HeHe01}. It is provided in Section~\ref{sec:proof clt}.

The sequence $(S_n(h))$ also verifies a restricted large deviation principle.
\medskip

\begin{theorem}[Restricted Large Deviation Principle]\label{thm:ldp empirical}
    Assume {\bf (Pur)} and {\bf (Erg)} hold. Then, for any $h\in C^\alpha(\sta,\mathbb R)$ and any probability measure $\nu$ on $\sta$, there exist $\theta_-<0<\theta_+$ and an analytic convex function $\Lambda:(\theta_-,\theta_+)\to \rr$ such that, 
    \begin{enumerate}
        \item $\forall \theta\in (\theta_-,\theta_+)$, $$\lim_{n\to \infty}\tfrac1n\log \ee_\nu(\exp(\theta S_n(h)))=\Lambda(\theta)=\log(\rho_\alpha(\Pi_\theta)).$$
        \item For any open subset $A$ of $(\partial_\theta^+\Lambda(\theta_-),\partial_\theta^-\Lambda(\theta_+))$,
        $$ \lim_{n\to\infty}\tfrac1n\log \pp_\nu\left(\tfrac1nS_n(h)\in A\right)= -\inf_{x\in {A}} I(x),$$
        where the rate function $I$ is given by
        $$I(x)=\sup_{\theta\in (\theta_-,\theta_+)} \theta x- \Lambda(\theta).$$
    \end{enumerate}
\end{theorem}
Given Theorem~\ref{Analytique}, the proof of the above theorem is a consequence of Kato's perturbation Theory \cite{kato2013perturbation} and G\"artner--Ellis Theorem (see \cite{dembo2009large}). It is provided in Section~\ref{sec:proof ldp}. Note that we call this theorem a restricted large deviation principle because it is restricted to sets $A$ in $(\partial_\theta^+\Lambda(\theta_-),\partial_\theta^-\Lambda(\theta_+))$ and we have no control on $\theta_\pm$. 


\subsubsection{Lyapunov exponent} \label{sec:results W_n} 

For any $n\in \nn$, recall that
$$W_n=V_n\dotsb V_1.$$
The quantity $\log\|W_n\|$ is called the Lyapunov exponent and is one of the main focus of research in random products of matrices -- see \cite{BL85} for example. In \cite[Proposition 4.3]{BFPP17}, in the case $E=\mathbb C^k$ it was proved that, if $\int \|v\|^2\log\|v\|\d \mu(v)<\infty$, there exists $\gamma\in[-\infty,\infty[$ such that
$$\lim_{n\to\infty}\tfrac1n\log \|W_nx_0\|=\lim_{n\to\infty}\tfrac1n\log \|W_n\|=\gamma$$
almost surely for any law of the initial state $\hat x_0$. Theorem~\ref{AnalyticLog} leads to finer limit theorems without the assumption $E=\mathbb C^k$.

First, a central limit theorem holds with Berry-Esseen like convergence bounds.
\begin{theorem}[Central Limit Theorem]\label{thm:clt log}
    Assume {\bf (Pur)} and {\bf (Erg)} hold. Assume there exist $\tau\in (0,2)$ and $\delta>0$ such that $\int \|v\|^{2-\tau}+\|v\|^{2+\delta}\d\mu(v)<\infty$.
    Then, there exists $\sigma^2 \geq 0$ such that for every probability measure $\nu$ on $\sta$,
    \begin{enumerate}
        \item Almost surely,
        $$\lim_{n\to\infty}\tfrac1n\log\|W_nx_0\|=\lim_{n\to\infty}\tfrac1n\log\|W_n\|=\int_{\mathrm{M}_d(\cc)\times\sta}\log\|vx\|\, \|vx\|^2\d\mu(v)\d\nu_{inv}(\hat x)=\gamma.$$
        \item For every bounded continuous function $\phi$,
        $$ \lim_{n \to \infty} \ee_{\nu} \left [\phi \left ( \frac{\log \|W_n x\|-n \gamma}{\sqrt{n}} \right ) \right ]=\lim_{n \to \infty} \ee_{\nu} \left [\phi \left ( \frac{\log \|W_n\|-n \gamma}{\sqrt{n}} \right ) \right ]=\mathcal N(0,\sigma^2)(\phi).$$
        \item If $\sigma^2 >0$, then there exists $C>0$ independent of $\nu$ such that
        $$ \underset{u \in \rr}{\sup} \; | \pp_{\nu}[\log \|W_n x\| -n \gamma \leq u \sigma \sqrt{n}] - \mathcal{N}(0,1)((-\infty,u])| \leq \frac{C}{\sqrt{n}}.$$
        and there exists $\tilde C>0$ independent of $\nu$ such that
        $$ \underset{u \in \rr}{\sup} \; | \pp_{\nu}[\log \|W_n\| -n \gamma \leq u \sigma \sqrt{n}] - \mathcal{N}(0,1)((-\infty,u])| \leq \frac{\tilde C}{n^{\frac14}}.$$
    \end{enumerate}
\end{theorem}
As for Theorem~\ref{thm:clt1}, this theorem is a corollary to Theorem~\ref{AnalyticLog} and is proved in Section~\ref{sec:proof clt}. The subtlety added compared to Theorem~\ref{thm:clt1} is passing from $\log\|W_nx\|$ to $\log\|W_n\|$. That is the reason why we only have a $n^{-\frac14}$ convergence rate for $\log \|W_n\|$. We do not believe this rate is optimal and a $n^{-\frac12}$ rate should hold. However, proving it implies studying the process $W_n/\|W_n\|$ directly and the related tilting of $\Pi$ and it is not within the scope of this article.

Second, a large deviation principle holds.
\begin{theorem}[Restricted Large Deviation Principle]\label{thm:ldp log}
    Assume {\bf (Pur)} and {\bf (Erg)} hold. Assume there exist $\tau\in (0,2)$ and $\delta>0$ such that $\int \|v\|^{2-\tau}+\|v\|^{2+\delta}\d\mu(v)<\infty$. Let $\alpha=1-\tau/2$.
    Then, for any probability measure $\nu$ on $\sta$, there exists $s_-<0<s_+$ and an analytic convex function $\Upsilon:(s_-,s_+)\to\rr$ such that,
    \begin{enumerate}
        \item $\forall s\in (s_-,s_+)$, 
        $$\lim_{n\to\infty} \tfrac1n \log\ee_\nu(\exp(s\log\|W_nx_0\|))=\Upsilon(s)=\log(\rho_\alpha(\Gamma_s)).$$
        \item For any open subset $A$ of $(\partial_s^+\Upsilon(s_-),\partial_s^-\Upsilon(s+))$,
        \begin{align*}
             \lim_{n\to\infty}\tfrac1n\log\pp_\nu(\tfrac1n\log\|W_nx\|\in A)= -\inf_{w\in {A}}J(w)
        \end{align*}
        where the rate function $J$ is given by
        $$J(w)=\sup_{s\in (s_-,s_+)}ws-\Upsilon(s).$$
    \end{enumerate}
    If moreover, $\ee_\nu(|\langle x,y\rangle|)=\int|\langle x,y\rangle|\d\nu(\hat x)>0$ for any $ y$, such that $\Vert y\Vert=1$ (that is the support of $\nu$ is not included in a hyperplane),
    \begin{enumerate}
        \item $\forall s\in (s_-,s_+)$, 
        $$\lim_{n\to\infty} \tfrac1n \log\ee_\nu(\exp(s\log\|W_n\|))=\Upsilon(s).$$
        \item For any open subset $A$ of $(\partial_s^+\Upsilon(s_-),\partial_s^-\Upsilon(s+))$,
        \begin{align*}
             \lim_{n\to\infty}\tfrac1n\log\pp_\nu(\tfrac1n\log\|W_n\|\in A)= -\inf_{w\in {A}}J(w).
        \end{align*}
    \end{enumerate}
\end{theorem}
As for Theorem~\ref{thm:ldp empirical}, this theorem is also a corollary to Theorem~\ref{AnalyticLog} and is proved in Section~\ref{sec:proof ldp}. As for Theorem~\ref{thm:clt log}, the subtlety added compared to Theorem~\ref{thm:ldp empirical} is passing from $\log\|W_nx\|$ to $\log\|W_n\|$.

\section{Proof of the spectral gap, Theorem~\ref{QCompact}}\label{sec:proof QCompact}
We first prove $\Pi$ is quasi-compact. It yields Items~(\ref{Compact1}) and (\ref{Compact1bis}) of Theorem~\ref{QCompact}. Along the way we prove $\rho_\alpha(\Pi)=1$ and thus Item~(\ref{Compact3}). Then, using quasi-compactness, we prove the peripheral spectrum is a finite group of simple eigenvalues. That yields Item~(\ref{Compact2}) of Theorem~\ref{QCompact}.
\subsection{Quasi-Compactness}
In this section we prove Items~(\ref{Compact1},\ref{Compact1bis},\ref{Compact3}) of Theorem~\ref{QCompact}.
\begin{definition}\label{DefQCompact}
Let $( \staalg , \| \cdot \|) $ a Banach space and  $Q$ a bounded operator on $\staalg$. The operator $Q$ is quasi-compact if there exists a decomposition of $\staalg$ into disjoint closed $Q$-invariant subspaces,
$$ \staalg = \mathcal{F} \oplus \mathcal{H}, $$
where $\mathcal F$ is a finite dimensional space, all the eigenvalues of $Q|_{\mathcal{F}}$ have modulus
$\rho(Q)$ and $\rho(Q|_{\mathcal{H}}) < \rho(Q)$ where $\rho$ denotes the spectral radius function.
\end{definition}

The main result of this section is the following theorem.
\begin{theorem}
    \label{thm:pur implique quasi-compact}
    Let $0<\alpha\leq 1$. Assume {\bf (Pur)} holds. Then the operator $\Pi$ is a quasi-compact operator on $(C^\alpha(\sta,\cc),\Vert.\Vert_\alpha)$. Moreover, $\rho_\alpha(\Pi)=1$.
\end{theorem}
Once this theorem is proved, Items~(\ref{Compact1}) and (\ref{Compact1bis}) of Theorem~\ref{QCompact} follow from the quasi-compactness definition. The equality $\rho_\alpha(\Pi)=1$, stated in Item~(\ref{Compact3}) of Theorem~\ref{QCompact}, requires a supplementary argument that we provide at the end of this section. Item~\eqref{Compact2} of Theorem~\ref{QCompact} will be proved in next section.

\medskip

The main tool we use to prove this theorem is Ionescu-Tulcea Marinescu's Theorem \cite{ITM50}. It gives very useful sufficient conditions to prove that an operator is quasi-compact. The following theorem is the version of \cite{HeHe01}.
\begin{theorem}[Theorem~XIV.3 in \cite{HeHe01}]\label{ITM1}
Let $( \staalg , \| \cdot \|) $ a Banach space and  $Q$ a bounded operator on $\staalg$.
Let $| \cdot |$ a continuous semi-norm on $\staalg$ and $Q$ a bounded operator on $( \staalg , \| \cdot \|) $ such that:
\begin{enumerate}
    \item \label{Hyp1} $Q( \{ f : f \in \staalg, \| f \| \leq 1 \} )$ is totally bounded in $(\staalg, | \cdot |)$.
    \item \label{Hyp2} There exists a constant $M$ such that $\text{for all } f \in \staalg,$
    $$ |Qf| \leq M |f|. $$
    \item \label{Hyp3} There exist $n \in \nn$ and real numbers $r,R$ such that $r < \rho(Q) \text{ and for all } f \in \staalg$,
    $$  \| Q^n f \| \leq R |f| + r^n \| f \|. $$
\end{enumerate}
Then $Q$ is quasi-compact.
\end{theorem}

Using this theorem, quasi-compactness of $\Pi$ is the consequence of the three following properties.
\begin{enumerate}[label=(\roman*)]
    \item  $\Pi( \{ f : f \in \staalg, \| f \|_\alpha \leq 1 \} )$ is relatively compact in $(C^\alpha(\sta,\cc),\Vert.\Vert_\infty)$.\label{it:Pi rel compact}
    \item  For all $f$ in $C^\alpha(\sta,\cc)$,
    $$ \Vert \Pi f\Vert_\infty \leq \Vert f\Vert_\infty.$$\label{it:Pi bounded}
    \item  There exist $n \in \nn$ and real numbers $r,R$ such that $r < \rho_\alpha(\Pi) \text{ and for all } f \in C^\alpha(\sta,\cc),$
    $$  \Vert \Pi^n f \Vert_\alpha \leq R \Vert f\Vert_\infty + r^n \| f \|_\alpha.$$\label{it:Pi geom conv}
\end{enumerate}
The norm $\|. \|_\infty$ plays the part of the semi-norm in Theorem~\ref{ITM1}.

\medskip
Item~\ref{it:Pi bounded} is a direct consequence of the fact that $\Pi$ is a Markov kernel. Items~\ref{it:Pi rel compact} and~\ref{it:Pi geom conv} require more arguments. The main one is an appropriate estimation of $\|\Pi^n\|_\alpha$. It relies on a sub-additive function introduced in \cite{BFPP17}.

On $\P(\cc^k)$, the metric $d(\hat x,\hat y)=\sqrt{1-|\langle x,y\rangle|^2}$ is naturally associated to the notion of exterior product. We recall the relevant definitions. For $x_1, x_2$ in $\cc^k$ we denote by $x_1\wedge x_2$ the alternating bilinear form $(y_1, y_2) \mapsto \det\big(\langle x_i, y_j\rangle \big)_{i,j=1}^2$. Then, the set of all $x_1\wedge x_2 $ is a generating family for the set $\wedge^2\cc^k$ of alternating bilinear forms on $\cc^k$, and we can define a Hermitian inner product by 
\[\langle x_1\wedge x_2, y_1\wedge y_2\rangle = \det\big(\langle x_i, y_j\rangle \big)_{i,j=1}^2, \]
and denote by $\|x_1\wedge x_2\|$ the associated norm. It is immediate to verify that
\begin{equation}\label{eq:d wedge}
\d(\hat{x},\hat{y})=\|x\wedge y\|=\|\pi_{\hat x} - \pi_{\hat y}\|=\tfrac12\|\pi_{\hat x}-\pi_{\hat y}\|_1,
\end{equation}
with $\pi_{\hat x}$ the orthogonal projector onto $\cc x$ and $\|.\|_1$ the trace norm.

For a linear map $A$ on $\mathbb C^k$, we write $\wedge^2 A$ for the linear map on $\wedge^2\cc^k$ defined by
\begin{equation}\label{eq_defwedgepA}
\wedge^2 A \,(x_1\wedge x_2)=Ax_1\wedge Ax_2.
\end{equation}

Let $g:\nn\to \rr_+$ be the function defined by
$$ g(n) =  \int_{\mat^n} \| \wedge^2 W_n \| \d \mu^{\otimes n}(v_1,...,v_n).$$
Remark that $\|\wedge A\|\leq \|A\|^2$ implies $g(n)\leq \left(\int \|v\|^2\d\mu(v)\right)^n$ for any $n\in \nn$.
The following lemma concerning $g$ is proved in \cite{BFPP17}.

\begin{lemma}[Lemma~3.6 in \cite{BFPP17}]\label{lem:g decreasing geom}
    The function $g$ is sub-multiplicative and if {\bf (Pur)} holds, there exist $n_0 \in \nn$ and $0< \lambda < 1 $ such that for all $n \geq n_0$, $g(n) \leq \lambda^n$.
\end{lemma}

{
Next lemma implies \Cref{prop:Endo}} 
and unlocks the proof of Items~\ref{it:Pi rel compact} and~\ref{it:Pi geom conv}. 
\begin{lemma}\label{Lemma}
For every $f \in C^{\alpha}(\sta,\cc) $ and every $n \in \nn$
$$ \|\Pi^n f\|_\alpha \leq 3k^{1-\alpha}g(n)^{\alpha} \| f \|_\alpha + (2k +1) \| f \|_{\infty}.$$
Hence, $\Pi$ is a bounded endomorphism of $C^{\alpha}(\sta,\cc)$.
\end{lemma}
\begin{proof}
For any $\hat x,\hat y\in \sta$, let
$$M(x,y) = \{ A\in \mat: \| Ax \| > \|Ay\| \}.$$
Fix $n \in \nn$, $\hat x,\hat y \in \sta$, and $f \in C^{\alpha}(\sta,\cc)$. We consider the following $(x,y)$ dependent decomposition 
$\mat = M(x,y) \cup M(y,x)\cup\{A\in\mathrm{M}_k(\mathbb C):\Vert Ax\Vert=\Vert Ay\Vert\}.$
By definition
$$\Pi^n f(\hat x)=\int_{M_k(\cc)^n}f(W_n \cdot \hat x) \| W_nx \|^2 d \mu^{\otimes n}.$$
Using the decomposition of $\mat$,
\begin{align}\label{eq:decompo Q+Q+R}
    \begin{split}
        \Pi^nf(\hat x) - \Pi^nf(\hat y)  =& \ Q_{x,y}(f)+Q_{y,x}(f) \\
     &+\int_{\|W_nx\|=\|W_ny\|} \big(f(W_n \cdot \hat x)  - f(W_n \cdot \hat y)\big) \| W_nx \|^2\d\mu^{\otimes n},
    \end{split}
\end{align}
with, for any $\hat x,\hat y \in \sta$,
$$Q_{x,y}(f)=\int_{W_n\in M(x,y)} f(W_n \cdot \hat x) \| W_nx \|^2 - f(W_n \cdot \hat y) \| W_ny \|^2 \d \mu^{\otimes n}.$$

We bound $Q_{x,y}(f)$ for any $\hat x,\hat y\in \sta$. For any $\hat x,\hat y\in \sta$,
\begin{align}
    Q_{x,y}(f)=&\int_{W_n\in M(x,y)}f(W_n\cdot \hat x)(\|W_nx\|^2-\|W_ny\|^2)\d\mu^{\otimes n}\label{eq:decompo Q f(w-w)}\\
            &+\int_{W_n\in M(x,y)} \big(f(W_n\cdot\hat x)-f(W_n\cdot \hat y)\big)\|W_n y\|^2 \d\mu^{\otimes n}\label{eq:decompo Q (f-f)w}.
\end{align}
Remark that $\|W_nx\|> \|W_ny\|$ implies $\|W_nx\|>0$ so that $W_n\cdot\hat y$ and $W_n\cdot \hat x$ are well defined with respect to the measure $\one_{W_n\in M(x,y)}\|W_ny\|^2\d\mu^{\otimes n}$.

We first tackle the bound on the right hand side of \eqref{eq:decompo Q f(w-w)}. For any $A\in \mat$, $\|Ax\|^2-\|Ay\|^2=\tr(A^*A(\pi_{\hat x}-\pi_{\hat y}))$. Then Hölder's inequality for Schatten norms on matrices implies 
$$|\|Ax\|^2-\|Ay\|^2|\leq \tr(A^*A)\|\pi_{\hat x}-\pi_{\hat y}\|.$$
Using $d(\hat x,\hat y)=\|\pi_{\hat x}-\pi_{\hat y}\|$, it follows,
\begin{align}\label{eq:esim f(norm-norm)}
    \begin{split}
        \Big|\int_{W_n\in M(x,y)}f(W_n\cdot \hat x)(\|W_nx\|^2&-\|W_ny\|^2)\d\mu^{\otimes n}\Big|\\
            &\leq \|f\|_\infty d(\hat x,\hat y)\int_{W_n\in M(x,y)} \tr(W_n^*W_n)\d\mu^{\otimes n}\\
            &=k\|f\|_\infty d(\hat x,\hat y)
    \end{split}
\end{align}
where we used $\int \tr(W_n^*W_n)\d\mu^{\otimes n}=k$ since $\int W_n^*W_n\d\mu^{\otimes n}=\id_{\cc^k}$.

We now deal with the bound on Equation~\eqref{eq:decompo Q (f-f)w}. Since $f$ is an $\alpha$-Hölder continuous function, for any $A\in \mat$ such that $\|Ax\|\|Ay\|>0$,
\begin{eqnarray*}
   |f(A\cdot \hat x) - f(A \cdot \hat y)| &\leq &m_\alpha(f)d(A\cdot \hat x,A\cdot \hat y)^\alpha\\&=& m_\alpha(f) \frac{\| \wedge^2 A\  x \wedge y \|^{\alpha}}{\| Ax \|^{\alpha}\| A y \|^{\alpha}}. 
\end{eqnarray*}
Then, the triangular inequality and $\|W_nx\|^{-\alpha}\leq \|W_ny\|^{-\alpha}$ imply
\begin{align*}
    \begin{split}
    \Big|\int_{W_n\in M(x,y)}\big(f(W_n\cdot\hat x)-&f(W_n\cdot \hat y)\big)\|W_n y\|^2 \d\mu^{\otimes n}\Big|\\
    &\leq m_\alpha(f) \int_{W_n\in M(x,y)} \|\wedge^2 W_n\ x\wedge y\|^\alpha \|W_n y\|^{2(1-\alpha)}\d\mu^{\otimes n}.
    \end{split}
\end{align*}
Using, $\|W_ny\|^2\leq \tr(W_n^*W_n)$ and denoting $\ppch$ the probability measure defined by $\d\ppch|_{\outalg_n}=\tfrac1k\tr(W_n^*W_n)\d\mu^{\otimes n}$ -- see \cite[Page~314]{BFPP17} -- it follows that,
\begin{align*}
    \begin{split}
    \Big|\int_{W_n\in M(x,y)}\big(f(W_n\cdot\hat x)-&f(W_n\cdot \hat y)\big)\|W_n y\|^2 \d\mu^{\otimes n}\Big|\\
    &\leq k^{1-\alpha}m_\alpha(f) \int_{W_n\in M(x,y)} \|\wedge^2 W_n\ x\wedge y\|^\alpha (\tfrac1k\tr(W_n^*W_n))^{1-\alpha}\d\mu^{\otimes n}\\
    &\leq k^{1-\alpha}m_\alpha(f)\ee_{\ppch}\left(\frac{\|\wedge^2 W_n\|^\alpha}{(\tfrac1k\tr(W_n^*W_n))^\alpha}\right) d(\hat x,\hat y)^\alpha.
    \end{split}
\end{align*}
Jensen's inequality applied to the concave function $x\mapsto x^\alpha$ implies,
\begin{align}\label{eq:estim (f-f)norm}
    \begin{split}
    \Big|\int_{W_n\in M(x,y)}\big(f(W_n\cdot\hat x)-&f(W_n\cdot \hat y)\big)\|W_n y\|^2 \d\mu^{\otimes n}\Big|\\
    &\leq k^{1-\alpha}m_\alpha(f)\left(\ee_{\ppch}\left(\frac{\|\wedge^2 W_n\|}{\tfrac1k\tr(W_n^*W_n)}\right)\right)^\alpha d(\hat x,\hat y)^\alpha\\
    &=k^{1-\alpha}m_\alpha(f) g(n)^\alpha d(\hat x,\hat y)^\alpha.
    \end{split}
\end{align}
A similar argumentation leads to the bound
\begin{align}\label{eq:estim (f-f)norm equal norm}
    \begin{split}
    \Big|\int_{\|W_nx\|=\|W_ny\|}\big(f(W_n\cdot\hat x)-&f(W_n\cdot \hat y)\big)\|W_n y\|^2 \d\mu^{\otimes n}\Big|\leq k^{1-\alpha}m_\alpha(f) g(n)^\alpha d(\hat x,\hat y)^\alpha
    \end{split}
\end{align}
for the third and last term of the right hand side of Equation~\eqref{eq:decompo Q+Q+R}.

Using Equations~\eqref{eq:estim (f-f)norm} and \eqref{eq:esim f(norm-norm)},
$$|Q_{x,y}(f)|\leq k^{1-\alpha}g(n)^\alpha \|f\|_\alpha d(\hat x,\hat y)^\alpha+k\|f\|_\infty d(\hat x,\hat y).$$
Since $d(\hat x,\hat y)\in [0,1]$, $d(\hat x,\hat y)\leq d(\hat x,\hat y)^\alpha$ and
\begin{align}\label{eq:bound on Q}
\frac{|Q_{x,y}(f)|}{d(\hat x,\hat y)^\alpha}\leq k^{1-\alpha}g(n)^\alpha \|f\|_\alpha +k\|f\|_\infty.
\end{align}
Applying this bound to $Q_{x,y}(f)$ and $Q_{y,x}(f)$, and the bound from Equation~\eqref{eq:estim (f-f)norm equal norm} to the first, second and third terms on the right hand side of Equation~\eqref{eq:decompo Q+Q+R} respectively, leads to,
$$\frac{|\Pi^n f(\hat x)-\Pi^n f(\hat y)|}{d(\hat x,\hat y)^\alpha}\leq 3 k^{1-\alpha}g(n)^\alpha \|f\|_\alpha +2k\|f\|_\infty.$$

Finally, $\|\Pi^n f\|_\infty\leq \|f\|_\infty$ and the definition of the Hölder norm yield the lemma.
\end{proof}

\begin{proof}[Proof of Theorem~\ref{thm:pur implique quasi-compact}]
    As already mentioned, it is sufficient to prove Items~\ref{it:Pi rel compact},~\ref{it:Pi bounded} and~\ref{it:Pi geom conv} and $\rho_\alpha(\Pi)=1$.

    For Item~\ref{it:Pi rel compact}, Lemma~\ref{Lemma} leads to the equicontinuity of the set $\Pi( \{ f \in C^\alpha(\sta,\cc): \| f \|_\alpha \leq 1 \} )$. Indeed for all $f\in  \{ f \in C^\alpha(\sta,\cc): \| f \|_\alpha \leq 1 \}$, since $\Vert f\Vert_\infty\leq \Vert f\Vert_\alpha$, we have
    \begin{align}\label{eq:uniform Holder coeff}
        \Vert \Pi f\Vert_\alpha\leq k^{1-\alpha}g(1)^\alpha+\left(2 \sqrt{2} k +1 \right)=K,
    \end{align}
    with $K$ independent of $f$. This yields the equicontinuity. 
    
    {
   It follows the image of the unit ball of $C^\alpha(\sta,\cc)$ by $\Pi$ is equicontinuous. Then, Arzelà–Ascoli Theorem implies the image of the unit ball of $C^\alpha(\sta,\cc)$ by $\Pi$ is relatively compact with respect to $\|.\|_\infty$ and \Cref{it:Pi rel compact} is proved.}
    
    As already mentioned, Item~\ref{it:Pi bounded} is a direct consequence of $\Pi$ being a Markov kernel.

    For Item~\ref{it:Pi geom conv}, first, $\Pi \one=\one$ implies $\rho_\alpha(\Pi)\geq 1$. Then, from Lemma~\ref{lem:g decreasing geom}, there exist $n_0\in \nn$ and $\lambda<1$ such that for $n\geq n_0$, $g(n)\leq \lambda^n$. Let $n\geq n_0$ be such that $r=(3k)^{\frac{1-\alpha}{n}}\lambda^\alpha<1$. Then, $3k^{1-\alpha}g(n)^\alpha\leq r^{n}<1$ and Lemma~\ref{Lemma} yields 
    $$\|\Pi^n f\|_\alpha\leq R\|f\|_\infty+ r^n\|f\|_\alpha$$
    with $r<1\leq \rho_\alpha(\Pi)$ and $R=2k+1$. 
    
    It remains to prove $\rho_\alpha(\Pi)\leq1$. 
    Lemmas~\ref{Lemma} and~\ref{lem:g decreasing geom} imply there exist $n_0\in \nn$ and $\lambda<1$ such that for any $n\geq n_0$,
    $$\|\Pi^n f\|_\alpha\leq 3k^{1-\alpha}\lambda^{\alpha n} \|f\|_\alpha+ (2k+1)\|f\|_\infty.$$
    Since $\|f\|_\infty\leq\|f\|_\alpha$,
    $$\|\Pi^n\|_\alpha^{\frac1n}\leq (3k^{1-\alpha}+2k+1)^{\frac{1}{n}}.$$
    Taking the limit $n\to\infty$ leads to $\rho_\alpha(\Pi)\leq 1$. Hence, $\rho_\alpha(\Pi)=1$ and the theorem is proved.
\end{proof}

\subsection{Peripheral spectrum}\label{sec:peripheral spec}

In \cite[\S 9]{He95}, the concepts of period and cycles are introduced for arbitrary quasi-compact Markov kernels. This article being in french and our model allowing us to explicitly prove the peripheral spectrum is a finite subgroup of $U(1)$ and simple, we prefer an explicit proof. It also has the upside of providing an explicit expression for the space $\mathcal F$ of Definition~\ref{DefQCompact}. The interested reader can prove the results, except the explicit description of $\mathcal F$, using \cite[\S 9]{He95}. {Note that related results appeared earlier in~\cite[\S~8]{krein1948linear}}.

First, since $\Pi$ is a contraction on $(C^0(\sta,\cc),\|.\|_\infty)$, its peripheral eigenvalues are semi-simple. This is expressed in the following lemma which is a consequence of the fact that $\sup_n\Vert \Pi^n\Vert<\infty$. In particular, the operator $\Pi$ is said to be of diagonal type \cite[Proposition III.1]{HeHe01}
\begin{lemma}\label{lem:semi-simple}
    Assume {\bf (Pur)} holds. Let $z\in  U(1)$ be an eigenvalue of $\Pi$. Then its algebraic and geometric multiplicities are equal.
\end{lemma}

Now we define cycles and period.
\begin{definition}[Cycles of $\mu$]\label{def:cycles}
    A $\ell$-cycle of $\mu$ is a set of orthogonal subspaces $\{E_1,\dotsc,E_\ell\}$ of $E$, such that, 
    $$E=E_1\oplus\dotsb \oplus E_\ell$$
    and for $\mu$-almost every $v$ and $i\in \{1,\dotsc,\ell\}$, $vE_i\subset E_{i+1}$ with $E_{\ell+1}=E_1$.
\end{definition}
\begin{definition}[Period of $\mu$]\label{def:period}
The period of $\mu$ denoted $m$ is the maximal $\ell$ such that there exists a $\ell$-cycle of $\mu$.
\end{definition}
This definition corresponds to the definition of period for a quantum channel -- see \cite[Definition~3.1]{BFPP17}.

Using these definitions, we can prove the peripheral eigenvalues form a finite subgroup of $U(1)$.
\begin{lemma}
    \label{lem:finite subgroup}
    Assume {\bf (Pur)} and {\bf (Erg)} hold. Let $m$ be the period of $\mu$. Then, the peripheral spectrum of $\Pi$ is the set of $m^{\text{th}}$ roots of unity:
    $$\{ z\in \operatorname{spec}\Pi:|z|=1\}=\{e^{i\frac{r}{m}2\pi},r=0,\ldots,m-1\}.$$
\end{lemma}
\begin{proof}
    From Theorem~\ref{thm:pur implique quasi-compact}, $\Pi$ is quasi-compact. Therefore, the elements of modulus $1$ of its spectrum are eigenvalues. Since $\Pi \one=\one$, $1\in \{ z\in \operatorname{spec}\Pi:|z|=1\}$. Let $z$ be an eigenvalue of $\Pi$ with modulus $1$ distinct from $1$ and $f$ an associated eigenvector. Then, from the convergence in distribution shown in \cite[Theorem~1.1]{BFPP17}, we have for any $\hat x\in \sta$,
    $$\lim_{n\to\infty} \tfrac1m\sum_{r=0}^{m-1}\Pi^{r+mn}f(\hat x)=\nu_{inv}(f),$$
    where $\nu_{inv}$ is the unique $\Pi$-invariant probability measure. Since $z\neq 1$ and $\Pi f=zf$, using $\nu_{inv}\Pi=\nu_{inv}$, it implies, $\nu_{inv}(f)=0$. Then, we also have
    $$\lim_{n\to\infty} \tfrac1m\sum_{r=0}^{m-1}\Pi^{r+mn}f(\hat x)=\lim_{n\to\infty} \left(z^{mn}\tfrac{1}{m}\sum_{r=0}^{m-1} z^{r}\right)f(\hat x)=0.$$
    Choosing $\hat x$ such that $f(\hat x)\neq 0$ and taking the limit along a subsequence such that $\lim_{k\to\infty} z^{mn_k}=1$ leads to
    $$\sum_{r=0}^{m-1} z^{r}=0.$$
    It follows that $z$ is an $m^{\text{th}}$ root of unity.

    It remains to show that all the $m^{\text{th}}$ roots of unity are in $\{ z\in \operatorname{spec}\Pi:|z|=1\}$.
    For $l\in \{1,\dotsc,m\}$, let $$f_l:\hat x\mapsto \sum_{r=1}^{m} e^{i\frac{rl}{m} 2\pi} \langle x,M_r x\rangle$$ where the positive semi-definite matrices $M_r$ are the ones defined in Proposition~\ref{prop:cycle operators}. Note that the functions $(f_l)$ are well defined since the value $\langle x,M_r x\rangle=\tr(M_r\pi_{\hat x})$ does not depend on the representative of $\hat x$. Furthermore, since $\langle x,Ax\rangle=\tr(A\pi_{\hat x})$, by Hölder's inequality for Schatten's matrix norms, for $A$ positive semi-definite, $\hat x\mapsto \langle x,Ax\rangle$ is $\tr(A)$-Lipschitz and therefore $\alpha$-Hölder. As finite sums of $\alpha$-Hölder functions, the functions $f_l$ are $\alpha$-Hölder. 
    
    Recall the definition of $\Phi$ as the map from $\mat$ to $\mat$ defined by
    $$\Phi:X\mapsto \int_{\mat} v^* X v\ \d\mu(v).$$
    Following Proposition~\ref{prop:cycle operators}, $\Phi(M_r)=M_{r-1}$ with $M_0=M_m$. By definition of $\Pi$, we have
    \begin{eqnarray*}
      \Pi f_l(\hat x)&=&\sum_{r=1}^{m}e^{i\frac{rl}{m} 2\pi} \int_{\textrm{M}_k(\cc)}\langle vx,M_r vx\rangle\Vert vx\Vert \d\mu(v)\\
      &=&\sum_{r=1}^{m}e^{i\frac{rl}{m} 2\pi} \int_{\textrm{M}_k(\cc)}\langle  x,v^* M_r v\, x\rangle \d\mu(v)\\&=& \sum_{r=1}^{m}e^{i\frac{rl}{m} 2\pi} \langle x,\Phi(M_r) x\rangle\\&=&\sum_{r=1}^{m}e^{i\frac{(r+1)l}{m} 2\pi} \langle x,M_r x\rangle.  
    \end{eqnarray*}
    Hence, $\Pi f_l=e^{i\frac{l}{m}2\pi} f_l$ and then $e^{i\frac{l}{m}2\pi}\in  \{ z\in \operatorname{spec}\Pi:|z|=1\}$ and the lemma is proved.
\end{proof}

It remains to prove the peripheral spectrum is simple.
\begin{lemma}\label{lem:charactrization of cF}
    Assume {\bf (Pur)} and {\bf (Erg)} hold. Let $m$ be the period of $\mu$. Then, the subspace $\mathcal F$ spanned by the eigenvectors of $\Pi$ corresponding to eigenvalues of modulus $1$ is the linear span of the functions
    $$\hat x\mapsto \langle x,M_r x\rangle, \quad r\in \{1,\dotsc,m\}$$
    where the positive semi-definite matrices $M_r$ are the ones defined in Proposition~\ref{prop:cycle operators}. 

    In particular $\dim \mathcal F=m$ and the peripheral spectrum of $\Pi$ is simple.
\end{lemma}
\begin{proof}
    From Lemma~\ref{lem:finite subgroup}, $\{ z\in \operatorname{spec}\Pi:|z|=1\}=\{e^{i\frac{r}{m}2\pi},r=0,\ldots,m-1\}$. Following Lemma~\ref{lem:semi-simple} all these eigenvalues are semi-simple. Then, quasi-compactness of $\Pi$ implies that $\Pi$ can be written as
    $$\Pi=\sum_{i=1}^{\dim(\mathcal F)}\lambda_iP_{\lambda_i}+T.$$
    In the above decomposition, the operators $P_{\lambda_i}$ are the Riesz projectors onto the eigenspaces of the eigenvalues $\lambda_i$ which belong to $\{e^{i\frac{r}{m}2\pi},r=0,\ldots,m-1\}$. We have $\rho_\alpha(T)<1$ and $TP_{\lambda_i}=P_{\lambda_i}T=0.$ Therefore we can define
    $\Pi_\infty=\lim_{n\to\infty} \Pi^{mn}$. This is the Riesz projector onto $\mathcal F$.

    Lemma~\ref{lem:convergence Pim} implies that for any function $f\in C^0(\sta,\cc)$,
    $$\Pi_\infty f(\hat x)=\frac1m\sum_{r=1}^m \langle x, M_r x\rangle \nu_r(f),$$
    where $\nu_r$ is the unique invariant probability measure of $\Pi^m$ such that $\nu_r(\{\hat x\in \sta: x\in E_r\})=1$. That yields the lemma.
\end{proof}
The combination of Lemmas~\ref{lem:finite subgroup} and~\ref{lem:charactrization of cF} imply Item~(\ref{Compact2}) of Theorem~\ref{QCompact} and we proved the other three statements in the previous section, so Theorem~\ref{QCompact} is proved.

\section{Proofs for tiltings}\label{sec:proofso}

\subsection{Proof of Theorem~\ref{Analytique}}\label{sec:proof Analytique}

The goal of this section is to prove Theorem~\ref{Analytique}. We follow the strategy of \cite{BL85}.

Let $z \in \cc$, $0< \alpha \leq 1$ and $h : \sta \to \mathbb R$ a $\alpha$-Hölder continuous function. Recall that the operator $\Pi_z$ is defined as follows. Let $f$ be a mesurable function and $\hat{x} \in \sta$, then
$$\Pi_z f(\hat{x}) = \int_{\mat} e^{z h(v \cdot \hat x)}f(v \cdot \hat x) \|vx\|^2 d \mu(v). $$
First, next proposition shows $\Pi_z$ is a bounded $C^\alpha(\sta,\cc)$ endomorphism.
\begin{proposition}\label{prop:piz}
For every $f \in C^\alpha(\sta,\cc)$, $\Pi_z f \in C^\alpha(\sta,\cc)$. Moreover, for any $z\in \cc$,
$$\sup_{f:\|f\|_\alpha=1}\Vert \Pi_z f\Vert_\alpha<\infty.$$
\end{proposition}
\begin{proof}
Remark that $\Pi_z f = \Pi e^{zh} f$. The function $\hat x\mapsto zh(\hat x)$ is bounded since $h$ is continuous and defined on a compact set. On any compact subset of $\cc$, $w\mapsto e^w$ is Lipschitz continuous. Hence, $\hat x\mapsto e^{zh(\hat x)}$ is $\alpha$-H\"older. Lemma~\ref{Lemma} implies
$$\|\Pi_z f\|_\alpha\leq 3k^{1-\alpha} g(1)^\alpha\|e^{z h}f\|_\alpha+(2k+1)e^{|\Re(z)|\|h\|_\infty}\|f\|_\infty.$$
Since the norm $\|. \|_\alpha$ is sub-multiplicative and $\|f\|_\alpha\geq \|f\|_\infty$, for any $f$ such that $\|f\|_\alpha=1$, we have
$$\|\Pi_z f\|_\alpha\leq 3k^{1-\alpha} g(1)^\alpha\|e^{z h}\|_\alpha+(2k+1)e^{|\Re(z)|\|h\|_\infty}$$
and the proposition is proved.
\end{proof}

Now that we have established that $\Pi_z$ is bounded for all $z\in\mathbb C$. We can attack the problem of analyticity which is the content of Theorem~\ref{Analytique}. To this end, we introduce the operators $\Pi_n$, defined as follows. For $n \in \nn$ and for $f : \sta \to \mathbb C$ a measurable function and $\hat{x} \in \sta$
$$ \Pi_n f(\hat{x}) = \int_{\mat} h(v \cdot \hat x)^n f(v \cdot \hat x) \|vx\|^2 d \mu(v). $$
The main point is then to show the expansion in series
$$\Pi_z = \sum_{n=0}^{\infty} \frac{z^n}{n!}\Pi_n,$$
with a convergence in norm. It is a consequence of the convergence of the series $\sum_{n=0}^{\infty} \frac{|z|^n}{n!} \|\Pi_n\|_\alpha$. Fubini's Theorem implies the equality of the series with $\Pi_z$.

\begin{proof}[Proof of Theorem~\ref{Analytique}]
Since $h$ and $f$ are $\alpha$-Hölder and bounded, $h^nf$ is $\alpha$-Hölder and bounded. From Lemma~\ref{Lemma}, using $\Pi_n f = \Pi h^n f$, we obtain,
\begin{align*}
    \|\Pi_n f\|_\alpha \leq 3k^{1-\alpha}g(1)^{\alpha} \| h^nf \|_\alpha + (2  k +1) \| h^n f \|_{\infty}.
\end{align*}
The sub-multiplicativity of the norms $\|.\|_\alpha$ and $\|.\|_\infty$ implies,
\begin{align*}
    \|\Pi_n f\|_\alpha \leq 3k^{1-\alpha}g(1)^{\alpha} \| h\|_\alpha^n\ \|f \|_\alpha + (2  k +1) \|h\|_\infty^n\  \|f\|_{\infty}.
\end{align*}
Since $\|f\|_\infty\leq \|f\|_\alpha$ and $x\mapsto x^n$ is increasing for non negative $x$,
\begin{align*}
    \|\Pi_n f\|_\alpha \leq \left(3k^{1-\alpha}g(1)^{\alpha} + (2  k +1) \right)\| h\|_\alpha^n\ \|f \|_\alpha.
\end{align*}
Then the series $\sum_{n=0}^{\infty} \frac{|z|^n}{n!} \|\Pi_n\|_\alpha$ is convergent for any $z\in \cc$ and the map $z \mapsto \Pi_z$ is analytic on $\cc$.   
\end{proof}

\subsection{Proof of Theorem~\ref{AnalyticLog}}\label{sec:proof AnalyticLog}
{
Recall that in this context, we assume there exist $\tau \in (0,2)$ and $\delta> 0$ such that $$\int_{\mat} \|v\|^{2 - \tau}d \mu(v) < + \infty\,\,\textrm{and}\,\, \int_{\mat}\|v\|^{2+\delta} d \mu(v) < + \infty.$$ 

Then for $z$ such that $\Re(z) \in (-\tau,\delta)$, we study $\Gamma_z$ defined by : 

\begin{equation}\label{eq:DefPreuveTH33}
     \Gamma_z f(\hat{x}) = \int_{\mat} f(v \cdot \hat x) e^{z\log\|v x\|}\|vx\|^{2} d \mu(v),
\end{equation}
for any continuous function $f$, and any $\hat x \in \sta$. When $\Vert vx\Vert=0$, we extend by continuity and set $e^{z\log\|vx\|}\|vx\|^2=0$. Throughout this section, except when stated otherwise, we set $\alpha=1-\tau/2$.

The proof of Theorem \ref{AnalyticLog} is based on \cite[Theorem 3.1]{arendt2000vector}.
\begin{theorem}[Theorem~3.1 in \cite{arendt2000vector}]\label{Th.Arrendt}
Let $\varphi : \mathfrak U \to X$ be a locally bounded function from an open set $\mathfrak U \subseteq \mathbb{C}$ to a complex Banach space $X$. Let $S'$ be a separating subset of the dual space $X'$,  which means that for each $x \in X\setminus\{0\}$, there exists $x' \in S'$ such that $x'(x) \neq 0$. 

If the scalar-valued function $x' \circ \varphi : \mathfrak U \to \cc$ is analytic for each $x' \in S'$, then $\varphi$ is analytic.
\end{theorem}

We intend to apply this theorem with  
$
\mathfrak{U} = \{z \in \mathbb{C} : \Re(z) \in (-\tau, \delta)\},$
$X = \mathcal{L}(C^\alpha(\sta, \mathbb{C}))$, the space of bounded endomorphisms of $C^\alpha(\sta, \mathbb{C})$, and  
$\varphi:z\mapsto \Gamma_z$. To prove local boundedness, we establish that $\Gamma_z$ is bounded on the appropriate space. The proof relies on the fact that, for $\Re(z)>0$ and any $\alpha\in (0,\min(1,\Re(z/2))$, the mapping  
\[
\hat{x} \mapsto \|Ax\|^z f(A \cdot \hat{x})
\]
is $\alpha$-Hölder for any $\alpha$-Hölder function $f$ with an appropriate bound on the $\|A\|$-dependent coefficient.
}
\begin{lemma}
    \label{lem:vxf(vx) holder}
    Let $z\in \cc$ be such that $\Re(z)>0$. Fix $\alpha\in (0,\min(1,\Re(z/2)))$ and $f\in C^\alpha(\sta,\cc)$. Then for any matrix $A\in \mat$, the function $f_A:\hat x\mapsto e^{z\log \|Ax\|}f(A\cdot \hat x)$ and $f_A(\hat x)=0$ whenever $Ax=0$, is $\alpha$-Hölder continuous with
    $$\|f_A\|_\alpha\leq 2^{\alpha^*}\frac{|z/2|}{\alpha^*}\|A\|^{\Re(z)}\|f\|_\alpha$$
    with $\alpha^*=\min(1,\Re(z/2))$.
\end{lemma}
\begin{proof}
    First, from the sup-norm submultiplicativity, $\|f_A\|_\infty\leq \|A\|^{\Re(z)}\|f\|_\infty$.

    Second, let $\hat x,\hat y\in \sta$ be distinct. We use the shorthand $t^z$ for $e^{z\log t}$. Without loss of generality we can assume $\|Ax\|\geq \|Ay\|$. Then, triangular inequality implies,
    \begin{align}\label{eq:split f_A}
        |f_A(\hat x)-f_A(\hat y)|\leq \left|(\|Ax\|^z-\|Ay\|^z)f(A\cdot \hat x)\right|+\|Ay\|^{\Re(z)}\left|f(A\cdot \hat y)-f(A\cdot \hat x)\right|
    \end{align}
    where $0\times f=0$ even when the argument of $f$ may be undefined.

    We shall apply Lemmas~\ref{lem:norm vx Lipschitz} and~\ref{vxz}. Namely, we apply Lemma~\ref{vxz} with $K=[0,\|A\|^2]$ and,  $\alpha^*$ and $\beta$, set to $\alpha^*=\min(1,\Re(z/2))$ and $\beta=\max(1,\Re(z/2))$ respectively,
   \begin{eqnarray*}
       |\|Ax\|^z-\|Ay\|^z|&=&|(\|Ax\|^{2})^{z/2}-(\|Ay\|^{2})^{z/2}|
       \\&\leq&\frac{|z/2|}{\alpha^*}\sup_{[0,\Vert A\Vert^2]}t^{\beta-1}|\|Ax\|^2-\|Ay\|^2|^{\alpha^*}\quad(\textrm{Lemma}\,\,~\ref{vxz})
       \\&\leq& 2^{\alpha^*}\frac{|z/2|}{\alpha^*}\|A\|^{2(\beta-1)}\Vert A\Vert^{2\alpha^*}d(\hat x,\hat y)^{\alpha^*} \quad(\textrm{Lemma}\,\,~\ref{lem:norm vx Lipschitz})\\
       &=&2^{\alpha^*}\frac{|z/2|}{\alpha^*}\|A\|^{\Re(z)}d(\hat x,\hat y)^{\alpha^*}. 
   \end{eqnarray*}
    Since $\alpha\leq \alpha^*$ and $d(\hat x,\hat y)\leq 1$, 
    \begin{align}
        \label{eq:bound (Axz-Ayz)f(Ax)}
        \left|(\|Ax\|^z-\|Ay\|^z)f(A\cdot \hat x)\right|\leq 2^{\alpha^*}\frac{|z/2|}{\alpha^*}\|A\|^{\Re(z)} \|f\|_\infty d(\hat x,\hat y)^\alpha.
    \end{align}

    For the second term on the right hand side of Equation~\eqref{eq:split f_A}, we can assume $\|Ay\|>0$ and therefore $\|Ax\|>0$. It follows from $d(A\cdot \hat x,A\cdot \hat y)=\frac{\|\wedge^2 A\ x\wedge y\|}{\|Ax\|\|Ay\|}$,
    $$\|Ay\|^{\Re(z)}\left|f(A\cdot \hat y)-f(A\cdot \hat x)\right|\leq \|Ay\|^{\Re(z)}m_\alpha(f) \frac{\|\wedge^2 A\|^\alpha}{\|Ax\|^\alpha\|Ay\|^\alpha}\|x\wedge y\|^\alpha.$$
    Since $\|Ax\|\geq \|Ay\|$, $t\mapsto t^\alpha$ is non decreasing, $\|x\wedge y\|=d(\hat x,\hat y)$ and $\|\wedge^2 A\|\leq \|A\|^2$,
    $$\|Ay\|^{\Re(z)}\left|f(A\cdot \hat y)-f(A\cdot \hat x)\right|\leq \|Ay\|^{\Re(z)-2\alpha}\|A\|^{2\alpha}m_\alpha(f)d(\hat x,\hat y)^\alpha.$$
    By definition of $\alpha$, $\Re(z)-2\alpha\geq 0$, thus $t\mapsto t^{\Re(z)-2\alpha}$ is non decreasing. Hence,
    \begin{align}
        \label{eq:bound Ayz(f(Ay)-f(Ax))}
        \|Ay\|^{\Re(z)}\left|f(A\cdot \hat y)-f(A\cdot \hat x)\right|\leq \|A\|^{\Re(z)} m_\alpha(f) d(\hat x,\hat y)^\alpha.
    \end{align}
    It follows from Equations~\eqref{eq:bound (Axz-Ayz)f(Ax)} and \eqref{eq:bound Ayz(f(Ay)-f(Ax))}, $m_\alpha(f_A)\leq \|A\|^{\Re(z)} (m_\alpha(f)+2^{\alpha^*}\frac{|z/2|}{\alpha^*}\|f\|_\infty)$. Since $2^{\alpha^*}\frac{|z/2|}{\alpha^*}\geq 1$, it follows,
    $$m_\alpha(f_A)\leq \|A\|^{\Re(z)}2^{\alpha^*}\frac{|z/2|}{\alpha^*}\|f\|_\alpha.$$
    and
    $$\|f_A\|_\infty\leq 2^{\alpha^*}\frac{|z/2|}{\alpha^*} \|A\|^{\Re(z)}\|f\|_\infty.$$
    Hence,
    $$\|f_A\|_\alpha\leq 2^{\alpha^*}\frac{|z/2|}{\alpha^*} \|A\|^{\Re(z)}\|f\|_\alpha$$
    and the lemma is proved.
\end{proof}
{ Now we can prove local boundedness.}
\begin{proposition}\label{Bounded}
    Let $\alpha =  1-\frac{\tau}{2}$. Then, for every  $z$ such that $\Re(z) \in (-\tau,\delta)$,  $\Gamma_z$ is bounded on $C^{\alpha}(\sta,\cc)$.  More precisely, 
  $$\|\Gamma_z\|_\alpha \leq 2^{\alpha^*}\frac{|1+z/2|}{\alpha^*} \int_{\mat} \|v\|^{\Re(z)+2} d \mu(v)$$
  with $\alpha^*=\min(1,1+\Re(z/2))$.
\end{proposition}
\begin{proof}
By Lemma~\ref{lem:vxf(vx) holder}, for $\mu$-almost every $v$, $f_v:\hat x\mapsto \|vx\|^{2+z}f(v\cdot \hat x)$ is $\alpha$-Hölder with
\begin{align}
    \label{eq:bound f_v}
    \|f_v\|_\alpha\leq 2^{\alpha^*}\frac{|1+z/2|}{\alpha^*}\|v\|^{2+\Re(z)}\|f\|_\alpha.
\end{align}
The triangular inequality implies,
$$\|\Gamma_z f\|_\alpha\leq \int_{\mat} \|f_v\|_\alpha\d\mu.$$
Hence, bound \eqref{eq:bound f_v} and the $\mu$-integrability of $v\mapsto \|v\|^{2+\Re(z)}$ yield the proposition.
\end{proof}
{We can conclude the proof of Theorem \ref{AnalyticLog}.
\begin{proof}[Proof of Theorem \ref{AnalyticLog}]
Let $\mathfrak U=\{z\in\mathbb C:\Re(z)\in(-\tau,\delta)\}$, $X=\mathcal L(C^\alpha(\sta,\cc))$ and $S' \subset X'$ the set of the functionals $\varphi_{f, \hat x} \in X'$ defined by the relation:
$$ \varphi_{f,\hat x} : \Psi \mapsto (\Psi f)(\hat x),$$
where $f \in C^\alpha(\sta,\cc)$ and $\hat x \in \sta$. Let $\Psi\in \mathcal L(C^\alpha(\sta,\cc))$ and assume $\varphi_{f,\hat x}(\Psi)=0$ for any $f\in C^\alpha(\sta,\cc)$ and $\hat x\in \sta$. Then $\Psi f=0$ for any $f\in C^\alpha(\sta,\cc)$, thus $\Psi=0$. It follows that $S'$ is separating. 

For every $f \in C^\alpha(\sta,\cc)$, $\hat x \in \sta$ and $v\in \supp\mu$, the function $z \mapsto f(v\cdot\hat x)e^{z\log\|vx\|}\|vx\|^2$ is analytic, $\mu$-measurable and dominated by a $z$ independent $\mu$-integrable function on $\mathfrak U$. Thus, by the theorem of holomorphy under the integral sign, $z\mapsto \Gamma_zf(\hat x)$ is analytic on $\mathfrak U$. Since, \Cref{Bounded} ensures $z \mapsto \Gamma_z$ is locally bounded, \Cref{Th.Arrendt} implies $z \mapsto \Gamma_z$ is analytic on $\{z\in\mathbb C:\Re(z)\in(-\tau,\delta)\}$.
\end{proof}
}
\section{Proofs of Limit Theorems}\label{sec:limitprooof}

\subsection{Proofs of Central Limit Theorems}\label{sec:proof clt}

The Berry-Esseen bounds in Theorems~\ref{thm:clt1} and~\ref{thm:clt log} are expressed without reference in the constant to the Hölder operator norm of $\nu$. This is a consequence of the following lemma.
\begin{lemma}\label{lem:holder norm probability}
    Let $\nu$ be a probability measure on $\sta$. Then,
    $$\|\nu\|_\alpha=\sup_{f:\|f\|_\alpha=1}\nu(f)=1.$$
\end{lemma}
\begin{proof}
    First, the constant function equal to $1$  (denoted $\mathbf{1}$) is $\alpha$-Hölder, so $\|\nu\|_\alpha\geq 1$.

    Second, since $\nu$ is non negative, we can restrict the supremum to non negative functions. Since $\|f\|_\alpha\geq \|f\|_\infty$, $\|f\|_\alpha=1$ implies $f\leq \one$ for any non negative function $f$. It follows $\nu(f)\leq 1$ for any non negative function such that $\|f\|_\alpha=1$. Hence, $\|\nu\|_\alpha\leq 1$ and the lemma is proved.
\end{proof}

\begin{proof}[Proof of Theorem~\ref{thm:clt1}]
Theorems~\ref{QCompact} and~\ref{Analytique} imply assumptions $\mathcal H[\infty]$ and $\widehat{\mathcal D}$ of \cite[Page~9]{HeHe01} hold with $Q=\Pi$, $\mathcal B=C^{\alpha}(\sta,\cc)$ and $\xi=h$. Then Theorem A and B in \cite{HeHe01} and Lemma~\ref{lem:holder norm probability} yield the theorem.
\end{proof}

Before proving Theorem~\ref{thm:clt log}, we shall need the following Lemma. 

\begin{lemma}\label{lemm:l1}
Assume there exist $\tau\in (0,2)$ and $\delta>0$ such that $\int \|v\|^{2-\tau}+\|v\|^{2+\delta}\d\mu(v)<\infty$. Then, 
$$\int\vert \log(\Vert v\Vert)\vert \Vert v\Vert^2\d\mu(v)<\infty.$$
In particular, for any probability measures $\nu$ on $\sta$
$$\mathbb E_{\mathbb P_\nu}\left[\vert\log(\Vert V_1 x_1\Vert)\vert\right]<\infty.$$
\end{lemma}

\begin{proof}
The first part follows from a direct function analysis proving there exists $C>0$ such that for all $x\geq 1,$ $\log x\leq Cx^\delta$ and for all $0<x\leq 1,$ $-\log x\leq Cx^{-\tau}$. In particular, it implies that $|\log(\Vert v\Vert)| \Vert v\Vert^2\leq C\left( \Vert v\Vert^{2-\tau}\mathbf 1_{0<\Vert v\Vert\leq 1}+\Vert v\Vert^{2+\delta}\mathbf 1_{\Vert v\Vert\geq 1}\right)$ and finally
$$\int\vert \log(\Vert v\Vert)\vert \Vert v\Vert^2\d\mu(v)\leq C\int \|v\|^{2-\tau}+\|v\|^{2+\delta}\d\mu(v).$$
The second part follows by the same argument since $$\mathbb E_{\mathbb P_\nu}\Big[\vert\log(\Vert V_1x_1\Vert)\vert\Big]=\int\vert\log(\Vert vx\Vert)\vert\Vert vx\Vert^2\d\nu(\hat x)\d\mu(v)$$
and 
\begin{eqnarray*}
  \vert\log(\Vert vx\Vert)\vert \Vert vx\Vert^2&\leq& C\left(\Vert vx\Vert^{2-\tau}\mathbf 1_{0<\Vert vx\Vert\leq 1}+\Vert vx\Vert^{2+\delta}\mathbf 1_{\Vert vx\Vert\geq 1}\right)\\&\leq& C\left(\Vert v\Vert^{2-\tau}+\Vert v\Vert^{2+\delta}\right)
\end{eqnarray*}
since $\Vert vx\Vert\leq \Vert v\Vert$ and $2-\tau>0$.
\end{proof}

\begin{proof}[Proof of Theorem~\ref{thm:clt log}]
 Let us start with Item (1). In \cite[Proposition~4.3]{BFPP17}, a similar result is proved but the assumption $E=\cc^k$ is required and the expression of $\gamma$ was not provided. Here, we shall use the perturbation theory to prove the result in full generality. We borrow the expression
$$\lim_{n\to\infty} \tfrac1n \log\ee_\nu(\exp(s\log\|W_nx_0\|))=\Upsilon(s)=\log(\rho_\alpha(\Gamma_s)),$$
from Theorem~\ref{thm:ldp log} with $\alpha=1-\tau/2$ and we aim to apply Theorem II.6.3 of \cite{ellis2006entropy}. The above convergence is also a result of the perturbation theory and is expressed for example in \cite[Lemma VIII.5]{HeHe01}. The differentiability of $\Upsilon$, required in \cite[Theorem II.6.3]{ellis2006entropy}, follows from the fact that the function $\Upsilon$ is actually analytic in a complex neighborhood of $0$. Indeed the analyticity of $z\rightarrow\rho_\alpha(\Gamma_z)$ comes from usual perturbation theory (one can consult \cite[Chapter 7, Section 1.3]{kato2013perturbation} or \cite[Lemma VIII.1]{HeHe01}). This way, since $\rho_\alpha(\Gamma_0)=1$, there exists a sufficiently small neighborhood $U$ of $0$ such that for all $z\in U, \rho_\alpha(\Gamma_z)\in\mathbb C\setminus {\mathbb R_-}$ and $z\rightarrow\log(\rho_\alpha(\Gamma_z))$ is analytic. Then Theorem II.6.3 of \cite{ellis2006entropy} implies that  there exists a constant $\gamma$, such that almost surely
$$\lim_{n\to\infty}\frac 1n \log(\|W_nx_0\|)=\gamma.$$
Now we have proved the existence of $\gamma$, let us give the expression. To this end, we shall use that $$\lim_{n\to\infty} \frac1n\log\|W_nx\|=\lim_{n\to\infty}\frac1n \sum_{k=1}^n\log\|V_k\hat x_{k}\|.$$
Note that the Markov chain $(V_n,\hat x_n)$ is $\mathbb P_{\nu_{inv}}$ invariant. Indeed we have, for all bounded and continuous functions $f$ on $\textrm{M}_k(\cc)\times\sta$:
\begin{eqnarray*}
&&\mathbb E_{\mathbb P_{\nu_{inv}}}\left( f(V_n,\hat x_n)     \right)\\&=&\int_{\mathrm M_k(\mathbb C)^n\times\sta}f(v_n,v_n v_{n-1}\ldots v_1\cdot \hat x)\Vert v_n\ldots v_1x\Vert^2 \d\mu^{\otimes n}(v_1,\ldots,v_n)\d\nu_{inv}(\hat x)\\
&=&\int_{\mathrm M_k(\mathbb C)^n}\int_{\sta}f(v_n,v_n\cdot (v_{n-1}\ldots v_1\cdot \hat x))\Vert v_n\ldots v_1x\Vert^2\d\nu_{inv}(\hat x)\d\mu^{\otimes n}(v_1,\ldots,v_n)\\
&=&\int_{\mathrm M_k(\mathbb C)^n}\int_{\sta}f(v_n,v_n\cdot (v_{n-1}\ldots v_1\cdot \hat x))\times\\&&\hphantom{cccccccccccccccccc}\times\left\Vert v_n\frac{v_{n-1}\ldots v_1 x}{\Vert v_{n-1}\ldots v_1 x\Vert}\right\Vert^2\Vert v_{n-1}\ldots v_1 x\Vert^2\d\nu_{inv}(\hat x)\d\mu^{\otimes n}(v_1,\ldots,v_n)\\
&=&\int_{\mathrm M_k(\mathbb C)^n}\int_{\sta}f(v_n,v_n\cdot (v_{n-1}\ldots v_1\cdot \hat x))\times\\&&\hphantom{cccccccccccccccccc}\times\Vert v_n(v_{n-1}\ldots v_1\cdot \hat x)\Vert^2\Vert v_{n-1}\ldots v_1 x\Vert^2\d\nu_{inv}(\hat x)\d\mu^{\otimes n}(v_1,\ldots,v_n)\\&=&
\int_{\mathrm M_k(\mathbb C)}\mathbb E_{\nu_{inv}}\left[f(v_n,v_n\cdot \hat x_{n-1})\Vert v_n x_{n-1}\Vert^2\right]\d\mu(v_n)\\
&=& \int_{\mathrm M_k(\mathbb C)}\mathbb E_{\nu_{inv}}\left[f(v_n,v_n\cdot \hat x_0)\Vert v_n x_0\Vert^2\right]\d\mu(v_n)\\&=&
\int_ {\mathrm M_k(\mathbb C)\times \sta}f(v_n,v_n\cdot \hat x)\Vert v_nx\Vert^2 \d\nu_{inv}(\hat x)\d\mu(v_n)\\
&=&\mathbb E_{\mathbb P_{\nu_{inv}}}\left( f(V_1,\hat x_1)     \right)
\end{eqnarray*}
 In line $6$, we have used the $\nu_{inv}$ invariance of $(\hat x_n)$. Since $(V_n,\hat x_n)$ is defined for $n\geq 1$, we have the desired invariance. Then Lemma~\ref{lemm:l1} allows to use the Birkhoff ergodic Theorem which implies that there exists an integrable random variable $X_\infty$ such that $\ee_{\nu_{inv}}(X_\infty)=\ee_{\nu_{inv}}(\log\|V_1 x_1\|)$ and $\mathbb P_{\nu_{inv}}$-almost surely 
$$\lim_{n\to\infty}\frac1n \sum_{k=1}^n\log\|V_n x_{n}\|=X_\infty.$$
Then, the definition of the constant $\gamma$ yields that $\mathbb P_{\nu_{inv}}$-almost surely, $X_\infty$ is a constant and $$\gamma=X_\infty=\ee_{\nu_{inv}}(X_\infty)=\mathbb E_{\mathbb P_{\nu_{inv}}}(\log\|V_1 x_1\|)=\int_{\mathrm{M}_d(\cc)\times\sta}\log\|vx\|\, \|vx\|^2\d\mu(v)\d\nu_{inv}( \hat x).$$ 

    Second, for Items~(2) and (3) for $\log\|W_nx\|$, the tilting $\Gamma_{it}$ cannot be expressed as it is in \cite{HeHe01}. However, in that reference, the only place where the expression of $\Gamma_{it}$ ($Q(t)$ in \cite{HeHe01}) matters is in Section~V, where $Q(t)(x,\d y)$ is assumed to be expressed as $Q(t)(x,\d y)=w_t(x,y)Q(x,\d y)$. That specific expression is only used in Lemma~V.9. Thankfully, the proof of this lemma adapts straightforwardly to our tilting defined by $\Gamma_{it}(\hat x,A)=\int \one_{A}(v\cdot \hat x)e^{i t\log \|vx\|}\|vx\|^2\d\mu(v)$. Hence, as in the proof of Theorem~\ref{thm:clt1}, using Theorems~\ref{QCompact} and~\ref{AnalyticLog}, we can apply Theorems A and B of \cite{HeHe01}, this time with $\xi(\hat x)=\log\|vx\|-\gamma$ and Lemma~\ref{lem:holder norm probability} yields Items~(2) and (3) for $\log\|W_nx\|$.

    Third, Item~(2) for $\log\|W_n\|$ is a consequence of Slutsky's Lemma and Item~(3) in \cite[Proposition~4.3]{BFPP17}. Indeed we have $\lim_{n\to \infty}\frac{\|W_nx\|}{\|W_n\|}=c(x)>0$ $\pp_\nu$-almost surely. Then, 
    $$\lim_{n\to\infty}\tfrac{1}{\sqrt{n}}(\log\|W_nx\|-\log\|W_n\|)=0,\quad \pp_\nu\as$$
    and Slutsky's Lemma yields the convergence in law of $\left(\frac{\log\|W_n\|-n\gamma}{\sqrt{n}}\right)$.

    Finally, Item~(3) for $\log\|W_n\|$ relies on a small modification of the proof in \cite[\S~VI.3]{HeHe01}. As usual, the proof of Berry-Esseen theorem relies on (Berry-)Esseen Lemma (\cite[Section 3.4.4]{durret2010probability} or \cite[Lemma~VI.3]{HeHe01}) bounding distance between cumulative distribution function using characteristic functions. We shall study the quantity
    $$ \underset{u \in \rr}{\sup} \; | \pp_{\nu}[\log \|W_n\| -n \gamma \leq u \sigma \sqrt{n}] - \mathcal{N}(0,1)((-\infty,u])|$$
    Let $\tau>0$. Using \cite[Lemma~VI.3]{HeHe01}, one has
    \begin{align*}
    \begin{split}
    &\underset{u \in \rr}{\sup} \; | \pp_{\nu}[\log \|W_n\| -n \gamma \leq u \sigma \sqrt{n}] - \mathcal{N}(0,1)((-\infty,u])|\\&\leq
    \frac{1}{\pi}\int_{-\sigma\tau\sqrt n}^{\sigma\tau\sqrt n}\frac{1}{\vert t\vert}\left\vert\mathbb E_\nu\left[e^{it\frac{\log(\Vert W_n\Vert)-\gamma}{\sqrt n} }\right]-e^{-\frac{t^2}{2}}\right\vert dt+ \frac{24}{\pi\sqrt 2}\frac{1}{\sigma\tau\sqrt n}
    \\&\leq\frac{1}{\pi}\int_{-\sigma\tau\sqrt n}^{\sigma\tau\sqrt n}\frac{1}{\vert t\vert}\left\vert\mathbb E_\nu\left[e^{it\frac{\log(\Vert W_n \Vert)-\gamma}{\sqrt n} }-e^{it\frac{\log(\Vert W_n x \Vert)-\gamma}{\sqrt n} }\right]\right\vert dt\\&\hphantom{\leq}+\frac{1}{\pi}\int_{-\sigma\tau\sqrt n}^{\sigma\tau\sqrt n}\frac{1}{\vert t\vert}\left\vert\mathbb E_\nu\left[e^{it\frac{\log(\Vert W_n x\Vert)-\gamma}{\sqrt n} }\right]-e^{-\frac{t^2}{2}}\right\vert dt+ \frac{24}{\pi\sqrt 2}\frac{1}{\sigma\tau\sqrt n}.
    \end{split}
\end{align*}
Using the element of proof of \cite[Lemma~VI.3]{HeHe01}, we have (for some constants $c>0$, $0<\eta<1$ and $C_3>0$)
\begin{eqnarray*}&&\frac{1}{\pi}\int_{-\sigma\tau\sqrt n}^{\sigma\tau\sqrt n}\frac{1}{\vert t\vert}\left\vert\mathbb E_\nu\left[e^{it\frac{\log(\Vert W_n x\Vert)-\gamma}{\sqrt n} }\right]-e^{-\frac{t^2}{2}}\right\vert dt+ \frac{24}{\pi\sqrt 2}\frac{1}{\sigma\tau\sqrt n}\\&\leq& \frac{c}{\pi\sigma\sqrt{n}}\left(2\sigma\tau\sqrt{n}((1-\eta)^n)+\int_{-\infty}^{+\infty}e^{\frac{-t^2}{4}}\d t\right)+\frac{C_3}{\sqrt n}+\frac{24}{\pi\sqrt 2}\frac{1}{\sigma\tau\sqrt n}.
\end{eqnarray*}
Indeed this is exactly the term that has to be handled to obtain Berry-Essen bound for $\log \|W_n x\|$. For the other term, it follows from Jensen's inequality and $|\sin(x)|\leq |x|$ that
    $$\left|\mathbb E_\nu\left[e^{it\frac{\log(\Vert W_n x\Vert)-\gamma}{\sqrt n} }-e^{it\frac{\log(\Vert W_n \Vert)-\gamma}{\sqrt n} }\right]\right|\leq \frac{|t|}{\sqrt n}\ee_\nu\left(\left|\log\frac{\|W_nx\|}{\|W_n\|}\right|\right).$$
    At this stage, since $\|W_nx\|\leq \|W_n\|$ and $x\mapsto -x^2\log x$ is bounded by $1$ on $[0,1]$, we have
    \begin{eqnarray*}\mathbb E_\nu\left(\left|\log\frac{\|W_nx\|}{\|W_n\|}\right|\right)&=&\int \left|\log\frac{\|W_n x\|}{\|W_n\|}\right|\|W_n x\|^2\d\mu(W_n) \d\nu(\hat x)\\&=&\int \left|\log\frac{\|W_n x\|}{\|W_n\|}\right|\frac{\|W_n x\|^2}{\|W_n\|^2}\|W_n\|^2\d\mu(W_n) \d\nu(\hat x)
    \\&\leq&\int \|W_n\|^2\d\mu(W_n).
    \end{eqnarray*}
  Now, it is clear that for any orthonormal basis $\{e_i,i=1,\ldots,k\}$ of $\cc^k$, $\|W_n\|^2\leq \sum_{i=1}^k\|W_n e_i\|^2$ and since $\int \|W_n y\|^2\d\mu(W_n)=1,$ for all $y$, we have
    \begin{align*}
      \mathbb E_\nu\left(\left|\log\frac{\|W_nx\|}{\|W_n\|}\right|\right)\leq k.  
    \end{align*}
    As a consequence, by gathering the terms with respect to their importance, we get
    $$\underset{u \in \rr}{\sup} \; | \pp_{\nu}[\log \|W_n\| -n \gamma \leq u \sigma \sqrt{n}] - \mathcal{N}(0,1)((-\infty,u])|\leq C_1\tau +\frac{C_2}{\tau\sqrt n}+\frac{C_3}{\sqrt n}.$$
    In particular $\tau=n^{-1/4}$ is the optimal choice and yields the desired bound.
\end{proof}

\subsection{Proofs of restricted Large Deviation Principles}\label{sec:proof ldp}

\begin{proof}[Proof of Theorem~\ref{thm:ldp empirical}]
    Theorems~\ref{QCompact} and~\ref{Analytique} imply Assumptions $\mathcal H[\infty]$ and $\widetilde{\mathcal D_{loc}}$ (Page~50) of \cite{HeHe01} hold. Then, since for $f=\one$, the measure $P_{\nu,n}$ defined in \cite[Page~53]{HeHe01} is $\nu$, \cite[Lemma~VI.5]{HeHe01} implies Item~(1) -- remark that in the proof of this lemma the assumptions $\nu(\xi)=0$ and $\sigma^2>0$ are not necessary. As a limit of convex functions, $\theta\mapsto \Lambda(\theta)$ is convex. The analyticity of $\Lambda$ is proved with the same arguments to the one used for $\Upsilon$ in the proof of Theorem~\ref{thm:clt log}.

    Then Item~(2) is a consequence of G\"artner-Ellis Theorem as formulated  in \cite[Theorem~A.5]{JOPP11} with $\mathcal I=\nn$, $M_n=\joint$, $\mathcal F_n=\mathcal J_n$ and $P_n=\pp_\nu$. That proves the theorem.
\end{proof}

The proof of the restricted LDP for $\log\|W_nx\|$ follows the same lines. The proof for $\log\|W_n\|$ relies on a lemma whose proof is adapted from \cite[Lemma~2.7]{Gui}.
\begin{lemma}\label{IneqLog}
    Let $\nu$ a probability measure on $\sta$ such that for any $y\in\cc^k\setminus\{0\}$, $\ee_\nu(|\langle y,x\rangle|)>0$. Then, for any $s>-2$, there exist $C_s(\nu)>0$ and $c_s(\nu)>0$ such that for any $v \in \mat$:
    \begin{equation} \label{GLPE}
        c_s(\nu) \|v\|^s \int_{\sta} \|vx\|^2\d\nu(\hat x)\leq  \int_{\sta} \|vx\|^{s+2} d \nu(\hat x) \leq C_s(\nu) \|v\|^s \int_{\sta} \|vx\|^2\d\nu(\hat x).
    \end{equation}
\end{lemma}
\begin{proof}
    Assume there exists $v$ such that $\int_\sta \|vx\|^2\d\nu(\hat x)=0$ or $\int_\sta \|vx\|^{2+s}\d\nu(\hat x)=0$. Then, $\nu$-almost every $x\in \ker v$. That implies $\ee_\nu(|\langle y, x\rangle|)=0$ for any $y\in \ker v^\perp$. That contradicts our assumption on $\nu$ therefore 
    $$\int_{\sta}\|vx\|^2\d\nu(\hat x)>0\quad \mbox{and}\quad \int_{\sta}\|vx\|^{s+2}\d\nu(\hat x)>0$$
    for any $v\in \mat$. It follows the function $$R_s:v\mapsto \frac{\int_{\sta} \|vx\|^{s+2} d \nu(\hat x)}{\|v\|^s\int_{\sta} \|vx\|^2 d \nu(\hat x)}$$ is continuous away from $v=0$. Since the inequalities we aim to prove are stable by division by $\|v\|^{s+2}$ we can assume $\|v\|=1$. The unit sphere of $\mat$ is compact and $R_s$ is continuous on this sphere. Therefore $C_s(\nu)=\sup_{v:\|v\|=1}R_s(v)<\infty$. Moreover, the minimum on the unit sphere, $c_s(\nu)$, is reached at some $v_0$. Since the function $R_s$ takes non-negative values, $c_s(\nu)\geq 0$. If $c_s(\nu)=0$, it implies $\int_{\sta}\|v_0x\|^{2+s}\d\nu(\hat x)=0$ which we already proved is impossible. Hence, $c_s(\nu)>0$.
\end{proof}

\begin{proof}[Proof of Theorem~\ref{thm:ldp log}]
    Since, $\ee_\nu(e^{s\log\|W_nx\|})=\nu\Gamma_s^n\one$, the proof for $\log\|W_nx\|$ is the same as the proof of Theorem~\ref{thm:ldp empirical}. We do not reproduce it.

    For $\log\|W_n\|$, it is sufficient to prove Item~(1) since Item~(2) follows using the same arguments as in the proof of Theorem~\ref{thm:ldp empirical}. Using Lemma~\ref{IneqLog}, for any $s>-2$, there exist $C_s(\nu)>0$ and $c_s(\nu)>0$ independent of $n$ such that, 
    $$c_s(\nu)\ee_\nu(e^{s\log\|W_n\|})\leq \ee_\nu(e^{s\log\|W_n x\|})\leq C_s(\nu)\ee_\nu(e^{s\log\|W_n\|}).$$
    Then it follows that
    $$\lim_{n\to\infty}\tfrac1n\log\ee_\nu(\exp(s\log\|W_n\|))=\lim_{n\to\infty}\tfrac1n\log\ee_\nu(\exp(s\log\|W_nx\|))$$
    for any $s>-2$ such that the second limit holds. Hence, Item~(1) for $\log\|W_nx\|$ implies Item~(1) for $\log\|W_n\|$ and the theorem is proved.
\end{proof}


\appendix

\section{Some lemmas}

\begin{lemma}
    \label{lem:norm vx Lipschitz}
    Let $A\in \mat$. Then $\hat x\mapsto \|Ax\|^2$ is $(2\|A\|^2)$-Lipschitz.
\end{lemma}
\begin{proof}
    By definition of $\pi_{\hat x}$, $\|Ax\|^2=\tr(A\pi_{\hat x} A^*)$. Hence, $\|Ax\|^2-\|Ay\|^2=\tr(A^*A(\pi_{\hat x}-\pi_{\hat y}))$. H\"older's inequality for matrix Schatten norms implies
    $$|\|Ax\|^2-\|Ay\|^2|\leq \|A\|^2\|\pi_{\hat x}-\pi_{\hat y}\|_1=2\|A\|^2d(\hat x,\hat y)$$
    with $\|.\|_1$ the trace norm and the lemma is proved.
\end{proof}

\begin{lemma}\label{vxz}
    Fix $z\in \cc$ such that $\Re(z)>0$. Let $\alpha^*=\min(1,\Re(z))$ and $\beta=\max(1,\Re(z))$. Then the function $t\mapsto e^{z\log t}$ continued in $0$, is $\alpha^*$-Hölder continuous on any compact subset $K$ of $\rr_+$ with Hölder coefficient bounded from above by
    $$\frac{|z|}{\alpha^*}\sup_{t\in K} t^{\beta-1}$$
    with $t^0=1$ for any $t\in \rr_+$.
\end{lemma}
\begin{proof}
    Let $t,s\in K$ be such that $t>s\geq  0$. For $s=0$, since $|e^{z\log t}|=t^{\Re(z)}$, if $\Re(z)\leq 1$, the lemma follows from $|z|/\Re(z)\geq 1$. If $\Re(z)>1$, $t^{\Re(z)}\leq  t\  \sup_{u\in K} u^{\Re(z)-1}$ and the lemma follows from $|z|>1$. 

We now assume $t>s>0$. Since $u\mapsto ze^{(z-1)\log u}$ is the derivative of $u\mapsto e^{z\log u}$, the fundamental theorem of calculus and triangular inequality imply 
$$|e^{z\log t}-e^{z\log s}|\leq |z|\int_s^t u^{\Re(z)-1} d u=\frac{|z|}{\Re(z)}(t^{\Re(z)}-s^{\Re(z)}).$$
For $\Re(z)\leq 1$ the lemma follows then from the $\Re(z)$-Hölder continuity of $x\mapsto x^{\Re(z)}$ with Hölder coefficient $1$. For $\Re(z)>1$, taking the supremum over $u\in K$ in the integral proves the lemma.
\end{proof}

\color{black}
\section{Cycles for quantum channels with unique invariant state}\label{app:cycl}
We investigate the consequence of our definition of period for $\mu$ on the eigenvectors of the map $\Phi:X\mapsto \int_{\mat} v^* Xv\d\mu(v)$. Recall the definition of cycles and period given in Section~\ref{sec:peripheral spec}. 
\begin{definition}\label{def:period app}
    A $\ell$-cycle of $\mu$ is a set of orthogonal subspaces $\{E_1,\dotsc,E_\ell\}$ of $E$, such that, 
    $$E=E_1\oplus\dotsb \oplus E_\ell$$
    and for $\mu$-almost every $v$ and $i\in \{1,\dotsc,\ell\}$, $vE_i\subset E_{i+1}$ with $E_{\ell+1}=E_1$.

    The period of $\mu$ denoted $m$ is the maximal $\ell$ such that there exists a $\ell$-cycle of $\mu$.
\end{definition}

\begin{proposition}\label{prop:cycle operators}
    Assume {\bf (Erg)} holds. Let $m$ be the period of $\mu$ as defined in Definition~\ref{def:period app}. Then there exist positive semi-definite matrices $\{M_r,r=1,\ldots,m\}$ such that
    $$\Phi(M_r)=M_{r-1}\quad \text{with } M_{0}=M_m$$
    and a m-cycle $\{E_1,\dotsc,E_m\}$ of $E$ such that
    $$M_r P_{E_r}=P_{E_r} M_r=P_{E_r}\quad\mbox{and}\quad M_rP_{E_{r'}}=P_{E_{r'}}M_r=0\quad \text{if }r\neq r'$$
    with $P_{E_r}$ the orthogonal projector onto $E_r$.

    Moreover their normalization can be chosen such that $\sum_{r=1}^m M_r=\id$.
\end{proposition}
\begin{proof}
    Let $\mathcal L(E)$ be the space of bounded endomorphisms of $E$. By assumption, the map $\Phi_E:\mathcal L(E)\to\mathcal L(E)$ defined by $\Phi_E^*(X)=\int_{\mat} v X_E v^*\d\mu(v)$, with $X_E$ the natural embedding of $X$ in $\mat$,is an irreducible trace preserving completely positive map. From Perron-Frobenius theory -- see \cite{EHK785} -- its spectral radius is $1$ and its peripheral spectrum $\{\lambda \in \operatorname{spec}\Phi_E:|\lambda|=1\}$ is simple and equal to $\{e^{i\frac{k}{m}2\pi},k=1,\ldots,m\}$ where $m$ is the period of $\mu$ as defined in Definition~\ref{def:period app}.

    By definition of $E$ in Assumption {\bf (Erg)}, for any density matrix $\rho\in \mat$, 
    $$\lim_{n\to\infty}P_{E^\perp}{\Phi^*}^n(\rho)P_{E^\perp}=0,$$
    with $P_{E^\perp}$ the orthogonal projector onto $E^\perp$. Indeed by invariance of $E$, we have $\Phi^*(P_E\rho P_E)=P_E \rho P_E,$ for all $\rho$. This way $P_{E^\perp}{\Phi^*}^n(\rho)P_{E^\perp}=P_{E^\perp}{\Phi^*}^n(P_{E^\perp}\rho P_{E^\perp})P_{E^\perp}$, for all $n$ and all $\rho$ which justifies the fact that the previous limit is equal to $0$. Indeed if it was not true since $P_{E^\perp}{\Phi^*}(P_{E^\perp}\rho P_{E^\perp})P_{E^\perp}$ is completely positive, there would exist a $\mu$-invariant subspace $E'\subset E^\perp$ and $E$ would not be unique. It follows that any eigenvector of $\Phi^*$ with eigenvalue of modulus $1$ is an element of $P_E\mat P_E$ with $P_E$ the orthogonal projector onto the space $E$. Hence, the peripheral spectrum of $\Phi^*$ and therefore $\Phi$ is also $\{e^{i\frac{k}{m}2\pi},k=1,\ldots,m\}$ with simple eigenvalues.

    From \cite[Theorem~6.16]{Wo12}, there exists $\{\rho_r,r=1,\ldots,m\}\subset P_E\mat P_E$ a set of density matrices such that $\operatorname{range}\rho_r=E_r$ where $E_r$ is a cyclic class of Definition~\ref{def:period app} and $\Phi^*(\rho_r)=\rho_{r+1}$ with $\rho_{m+1}=\rho_1$. It follows that the eigenspace of ${\Phi^*}^m$ associated to the eigenvalue $1$ is spanned by the orthogonal basis $\{\rho_r,r=1,\ldots,m\}$. Moreover, $1$ is the only eigenvalue of ${\Phi^*}^m$ with modulus $1$. Hence, there exist matrices $\{M_r,r=1,\ldots,m\}$ such that
    \begin{align}\label{eq:conv Phim}
    \Phi^*_{|1|}(\rho):=\lim_{n\to\infty} {\Phi^*}^{mn}(\rho)=\sum_{r=1}^m \rho_r \tr(M_r\rho),
    \end{align}
    for any density matrix $\rho$. By positivity of $\Phi^*$, each $M_r$ is positive semi-definite. Moreover, since $\Phi^*$ is trace preserving, $\sum_{r=1}^m \tr(M_r\rho)=1$ for any density matrix $\rho\in\mat$. Hence, $\sum_{r=1}^m M_r=\id$.

    Using that $v^mE_r\subset E_r$ for $\mu$-almost all $v$, on the one hand, $\tr(M_r\rho)=1$ for any density matrix $\rho\in \mat$ such that $\operatorname{range} \rho\subset E_r$. On the other hand, $\tr(M_r\rho)=0$ for any density matrix $\rho\in\mat$ such that $\operatorname{range} \rho\subset E_{r'}$ for some $r'\neq r$. Hence $P_{E_r}M_rP_{E_r}=P_{E_r}$ and $P_{E_{r'}}M_rP_{E_{r'}}=0$ for $r'\neq r$. Then, since $M_r$ is semi-definite and $\sum_{r=1}^m M_r=\id$, $P_{E_r}M_r=M_rP_{E_r}=P_{E_r}$ and $P_{E_{r'}}M_r=M_rP_{E_{r'}}=0$ for $r'\neq r$.

    It remains to prove $\Phi(M_r)=M_{r-1}$. From Equation~\eqref{eq:conv Phim} and $\Phi^*(\rho_r)=\rho_{r+1}$ with $\rho_{m+1}=\rho_1$, we deduce
    $$\Phi(X)=\sum_{r=1}^m M_r\tr(\rho_{r+1} X) + R(X)$$
    with $\lim_{n\to\infty} \|R^n\|^{\frac1n}<1$. It follows that $\Phi(M_r)=M_{r-1}$ with $M_0=M_m$ and the proposition is proved.
\end{proof}
We prove an auxiliary result on the $m$ time steps Markov chain defined by $\Pi^m$. The proof relies on the results of \cite{BFPP17}. This result can be seen as an addendum to \cite{BFPP17}.
\begin{lemma}
    \label{lem:convergence Pim}
    Assume {\bf (Pur)} and {\bf (Erg)} hold. Then there exist a unique set of probability measures $\{\nu_r,r=1,\ldots,m\}$ such that $\nu_r(\mathrm P(E_r))=1$ and $\nu_r\Pi^m=\nu_r$. Moreover, for any function $f\in C^0(\sta,\cc)$,
    $$\lim_{n\to\infty}\Pi^{mn}f(\hat x)=\sum_{r=1}^m \langle x, M_r x\rangle \nu_r(f),$$
    where the matrices $M_r$ are the ones defined in Proposition~\ref{prop:cycle operators}.
\end{lemma}
\begin{proof}
    From \cite[Proposition~3.5]{BFPP17}, {\bf (Pur)} implies there exist $(\hat y_n)$ an $\outalg$-measurable sequence of random variables taking value in $\sta$ such that there exist $C>0$ and $\lambda\in(0,1)$ such that for any probability measure $\nu$ on $\sta$ and $l,n\in \nn$,
    \begin{align}\label{eq:estimation hatx}
        \ee_\nu(d(\hat x_{m(n+l)},\hat y_{mn}\circ \theta^{ml}))\leq C\lambda^n
    \end{align}
    with $\theta$ the left shift $\theta(v_1,v_2,\dotsc)=(v_2,v_3,\dotsc)$.

    Following Equation~(22) in the proof of \cite[Proposition~3.4]{BFPP17}, for any $\outalg$-measurable bounded function $f$ and any density matrix $\rho\in\mat$,
    $$\ee^\rho(f\circ\theta)=\ee^{\Phi^*(\rho)}(f)$$
    with $\ee^\rho$ the expectation with respect to the probability measure defined, using Kolmogorov extension theorem, by
    $$d\pp^\rho(v_1,v_2,\dotsc)=\tr(\rho v_1^*v_2^*\dotsb v_n^*v_n\dotsb v_2v_1)\d\mu^{\otimes n}(v_1,v_2,\dotsc)$$
    See \cite[\S~2]{BFPP17} for details.

    Then it follows from Equation~\eqref{eq:conv Phim} that
    \begin{align}\label{eq:conv shift}
        \lim_{n\to\infty}\ee^\rho(f\circ\theta^{ml})=\ee^{\Phi^*_{|1|}(\rho)}(f)
    \end{align}
    for any bounded $\outalg$-measurable $f$.

    As in the proof of \cite[Theorem~1.1]{BFPP17}, combining the uniform  convergence in $l$ of Equation~\eqref{eq:estimation hatx} and the convergence in $l$ of Equation~\eqref{eq:conv shift}, for any $f:\sta\to \cc$ continuous,
    $$\lim_{n\to\infty} \ee_{\delta_{\hat x}}(f(\hat x_{mn}))-\sum_{r=1}^m \langle x,M_r x\rangle \ee^{\rho_r}(f(\hat y_{mn}))=0.$$
    From \cite[Appendix~B]{BFPP17} applied to $\Pi^m$,
    $$\lim_{n\to\infty}\ee^{\rho_r}(f(\hat y_{mn}))=\ee_{\nu_r}(f(\hat x))$$
    with $\nu_r$ the unique $\Pi^m$-invariant probability measure such that $\ee_{\nu_r}(\pi_{\hat x})=\rho_r$ with $\pi_{\hat x}$ the orthogonal projection onto $\cc x$. That proves the lemma.
\end{proof}

\paragraph{\bf Acknowledgements} The authors were supported by the ANR project ``ESQuisses", grant number ANR-20-CE47-0014-01 and by the ANR project ``Quantum Trajectories'' grant number ANR-20-CE40-0024-01. C. P. is also supported by the ANR projects Q-COAST ANR-19-CE48-0003. C. P. and T. B. are also supported by the program ``Investissements d'Avenir'' ANR-11-LABX-0040 of the French National Research Agency. 
{The authors would like to thank the anonymous referee who point out \cite[Theorem 3.1]{arendt2000vector} and the related simplification of \Cref{AnalyticLog} proof.}

\bibliographystyle{amsalpha}

\newcommand{\etalchar}[1]{$^{#1}$}
\providecommand{\bysame}{\leavevmode\hbox to3em{\hrulefill}\thinspace}
\providecommand{\MR}{\relax\ifhmode\unskip\space\fi MR }
\providecommand{\MRhref}[2]{%
  \href{http://www.ams.org/mathscinet-getitem?mr=#1}{#2}
}
\providecommand{\href}[2]{#2}

\end{document}